\documentclass[12pt
]{article}

\usepackage[margin=0.97in]{geometry}
\usepackage{color}
\usepackage{amsmath,amsthm,amssymb,mathtools,bbm,mathrsfs}
\usepackage[authoryear,longnamesfirst]{natbib}
\usepackage{array}
\usepackage{enumitem}
\usepackage{textcase}

\usepackage[breaklinks=true,hidelinks]{hyperref}

\numberwithin{equation}{section}
\allowdisplaybreaks[4]

\newtheoremstyle{plain2}{\topsep}{2\topsep}{\itshape}
{0pt}{\bfseries}{.}{.5em}{}
\newtheoremstyle{definition2}{\topsep}{2\topsep}{}
{0pt}{\bfseries}{.}{.5em}{}

\theoremstyle{plain2}
\newtheorem{theorem}{Theorem}[section]

\newtheorem{lemma}[theorem]{Lemma}
\newtheorem{corollary}[theorem]{Corollary}

\theoremstyle{definition2}

\newtheorem{remark}[theorem]{Remark}

\makeatletter

\renewcommand{\cite}{\citet}

\def\^#1{\ifmmode {\mathaccent"705E #1} \else {\accent94 #1} \fi}
\def\~#1{\ifmmode {\mathaccent"707E #1} \else {\accent"7E #1} \fi}

\def\*#1{#1^\ast}
\edef\-#1{\noexpand\ifmmode {\noexpand\bar{#1}} \noexpand\else \-#1\noexpand\fi}
\def\>#1{\vec{#1}}
\def\.#1{\dot{#1}}
\def\wh#1{\widehat{#1}}
\def\wt#1{\widetilde{#1}}
\def\atop{\@@atop}
\def\%#1{\mathcal{#1}}

\let\original@left\left
\let\original@right\right
\renewcommand{\left}{\mathopen{}\mathclose\bgroup\original@left}
\renewcommand{\right}{\aftergroup\egroup\original@right}

\renewcommand{\phi}{\varphi}
\newcommand{\eps}{\varepsilon}

\newcommand{\eq}{\eqref}

\newcommand{\dtv}{\mathop{d_{\mathrm{TV}}}\mathopen{}}

\newcommand{\dloc}[1][]{\mathop{d_{\mathrm{loc}}^{\,#1}}\mathopen{}}

\newcommand{\bigo}{\mathop{\mathrm{{}O}}\mathopen{}}

\def\ER{Erd\H{o}s--R\'enyi}
\newcommand{\I}{\mathop{{}\mathrm{I}}}
\newcommand{\Po}{\mathop{\mathrm{Po}}}

\newcommand{\TP}{\mathop{\mathrm{TP}}}

\newcommand{\Bi}{\mathop{\mathrm{Bi}}}

\newcommand{\Be}{\mathop{\mathrm{Be}}}

\newcommand{\Var}{\mathop{\mathrm{Var}}}

\newcommand{\law}{\mathscr{L}}

\def\be#1{\begin{equation*}#1\end{equation*}}
\def\ben#1{\begin{equation}#1\end{equation}}
\def\bes#1{\begin{equation*}\begin{split}#1\end{split}\end{equation*}}
\def\besn#1{\begin{equation}\begin{split}#1\end{split}\end{equation}}

\def\ba#1{\begin{align*}#1\end{align*}}

\def\given{\typeout{Command 'given' should only be used within bracket command}}
\newcounter{@bracketlevel}
\def\@bracketfactory#1#2#3#4#5#6{
\expandafter\def\csname#1\endcsname##1{%
\addtocounter{@bracketlevel}{1}%
\global\expandafter\let\csname @middummy\alph{@bracketlevel}\endcsname\given%
\global\def\given{\mskip#5\csname#4\endcsname\vert\mskip#6}\csname#4l\endcsname#2##1\csname#4r\endcsname#3%
\global\expandafter\let\expandafter\given\csname @middummy\alph{@bracketlevel}\endcsname
\addtocounter{@bracketlevel}{-1}}%
}
\def\bracketfactory#1#2#3{%
\@bracketfactory{#1}{#2}{#3}{relax}{1mu plus 0.25mu minus 0.25mu}{0.6mu plus 0.15mu minus 0.15mu}
\@bracketfactory{b#1}{#2}{#3}{big}{1mu plus 0.25mu minus 0.25mu}{0.6mu plus 0.15mu minus 0.15mu}
\@bracketfactory{bb#1}{#2}{#3}{Big}{2.4mu plus 0.8mu minus 0.8mu}{1.8mu plus 0.6mu minus 0.6mu}
\@bracketfactory{bbb#1}{#2}{#3}{bigg}{3.2mu plus 1mu minus 1mu}{2.4mu plus 0.75mu minus 0.75mu}
\@bracketfactory{bbbb#1}{#2}{#3}{Bigg}{4mu plus 1mu minus 1mu}{3mu plus 0.75mu minus 0.75mu}
}
\bracketfactory{clc}{\lbrace}{\rbrace}
\bracketfactory{clr}{(}{)}
\bracketfactory{cls}{[}{]}
\bracketfactory{abs}{\lvert}{\rvert}
\bracketfactory{norm}{\Vert}{\Vert}
\bracketfactory{floor}{\lfloor}{\rfloor}
\bracketfactory{ceil}{\lceil}{\rceil}
\bracketfactory{angle}{\langle}{\rangle}

\newcounter{ctr}\loop\stepcounter{ctr}\edef\X{\@Alph\c@ctr}%
	\expandafter\edef\csname s\X\endcsname{\noexpand\mathscr{\X}}
	\expandafter\edef\csname c\X\endcsname{\noexpand\mathcal{\X}}
	\expandafter\edef\csname b\X\endcsname{\noexpand\boldsymbol{\X}}
	\expandafter\edef\csname I\X\endcsname{\noexpand\mathbbm{\X}}
\ifnum\thectr<26\repeat

\newcount\minute
\newcount\hour
\newcount\hourMins
\def\now{%
\minute=\time%
\hour=\time \divide \hour by 60%
\hourMins=\hour \multiply\hourMins by 60%
\advance\minute by -\hourMins%
\zeroPadTwo{\the\hour}:\zeroPadTwo{\the\minute}%
}
\def\zeroPadTwo#1{\ifnum #1<10 0\fi#1}

\renewcommand\section{\@startsection {section}{1}{\z@}%
{-3.5ex \@plus -1ex \@minus -.2ex}%
{1.3ex \@plus.2ex}
{\center\small\sc\mathversion{bold}\MakeTextUppercase}}

\def\subsection#1{\@startsection {subsection}{2}{0pt}%
{-3.5ex \@plus -1ex \@minus -.2ex}%
{1ex \@plus.2ex}%
{\bf\mathversion{bold}}{#1}}

\def\subsubsection#1{\@startsection{subsubsection}{3}{0pt}%
{\medskipamount}%
{-10pt}%
{\normalsize\itshape}{\kern-2.2ex. #1.}}

\def\blfootnote{\xdef\@thefnmark{}\@footnotetext}

\makeatother

\def\giv{\,|\,}

\def\a{\alpha}
\def\b{\beta}
\def\s{\sigma}
\def\d{\delta}

\def\law{{\mathcal L}}
\def\cV{{\mathcal V}}

\def\ignore#1{}
\def\tW{{\widetilde W}}

\def\hW{{\widehat W}}

\def\uii{^{(i)}}

\def\tG{{\widetilde G}}

\def\tV{{\widetilde V}}
\def\l{\lambda}
\def\Eq{\ =\ }
\def\Le{\ \le\ }
\def\Def{\ :=\ }

\def\nin{\noindent}

\def\tD{{\widetilde D}}

\def\delloc{\d_{\mathrm{loc}}}
\def\Ref{\eqref}
\def\m{\mu}
\def\ex{\mbE}
\def\pr{\mbP}
\def\non{\nonumber}
\def\bone{\mathbf{1}}
\def\hcI{\widehat{\mathcal{I}}}
\def\leb{{\rm Leb}}
\def\card{{\rm card}}

\def\tD{{\widetilde D}}

\def\url#1{{\tt #1}}

\usepackage{dsfont}

\newcommand{\registered}
   {{\scriptsize \ooalign{\hfil\raise0.07ex\hbox{\scriptsize \sc r}\hfil%
              \crcr\mathhexbox20D}}}

\newcommand{\nc}{\newcommand}
\nc{\ds}{\displaystyle}
\nc{\mbZ}{\mathbb Z}
\nc{\mbQ}{\mathbb Q}
\nc{\mbR}{\mathbb R}
\nc{\mbC}{\mathbb C}
\nc{\mbN}{\mathbb N}
\nc{\mbE}{\mathbb E}
\nc{\mbP}{\mathbb P}

\nc{\PH}{\emph{PH} }
\nc{\ME}{\emph{ME} }
\nc{\LST}{\emph{LST} }
\nc{\rank}{\mbox{rank\hspace{1pt}}}

\newcommand{\Wdeq}{\wh W_d}

\begin{document}

\title{\sc\bf\large\MakeUppercase{Local limit theorems for occupancy models}}
\author{\sc A. D. Barbour\thanks{Universit\"at Z\"urich; {\tt a.d.barbour@math.uzh.ch}},
Peter Braunsteins\thanks{University of Melbourne; {\tt braunsteins.p@unimelb.edu.au}},
Nathan Ross\thanks{University of Melbourne; {\tt nathan.ross@unimelb.edu.au}}
}
\maketitle

\begin{abstract}  
We present a rather general method for proving local limit theorems, with a good rate of convergence, for sums of
dependent random variables.  The method is applicable when a Stein coupling can be exhibited.    
Our approach involves both Stein's method for distributional approximation and Stein's method for concentration.
As applications, we
prove local central limit theorems with rate of convergence for the number of germs with~$d$ neighbours in a germ--grain 
model, and the number of degree-$d$ vertices in an \ER\ random graph.  In both cases, the error rate is optimal,
up to logarithmic factors.
\end{abstract}


\section{Introduction}

Local central limit theorems (LCLTs) for sums of independent random variables have been well studied,
largely using characteristic function techniques; see \cite[Chapter~{VII}.1]{Pet12}. For the standard example, if $X_1, X_2, \dots$ 
are i.i.d.\ aperiodic integer valued 
random variables with finite third moment, and $W := \sum_{i=1}^n X_i$, with $\mbE W = \mu$ and $\Var(W)=\sigma^2$, 
then 
\begin{equation}\label{iidPetrov}
   \sup_{k} \left| \mbP(W = k) - \frac{1}{\sigma \sqrt{2\pi}} \exp \left\{ -\frac{1}{2\sigma^2}(k-\mu)^2 \right\} \right|
         \Eq \bigo(1/\sigma^2).
\end{equation}
When the summands $X_i$ are dependent, there are few general methods available for proving LCLTs with error bounds.
In this paper, we combine results from \cite{Rol15}, \cite{Bar18} and \cite{Bar19} to present an approach that 
is quite widely applicable. We illustrate its power in the context of certain \emph{occupancy models}.

Our random variables of interest take the form
\begin{equation}\label{Occd}
   \Wdeq \Def  \sum_{i=1}^n \I\{M_i = d\}, \quad d \geq 0,
\end{equation}
where $n$ is the number of \emph{locations} and $M_i$ is the number of \emph{occupants} at 
 location $i \in [n] :=\{1, \dots,n \}$ in an occupancy model.  Examples of occupancy models include
\begin{enumerate}
\item \emph{multinomial occupancy models}, where $m$ balls (the occupants) are placed in~$n$ urns (the locations), 
     independently at random, and~$M_i$ is the number of balls in urn $i$;
\item \emph{\ER\ random graphs}, where edges (the occupants) between distinct 
pairs of~$n$ vertices (the locations) are independently present or absent, with a common probability~$p$, 
and~$M_i$ is the degree of vertex~$i$;
\item \emph{germ--grain models\/}, where~$n$ points (the occupants), which form the centres of balls of fixed radius 
(the locations), are placed uniformly at random in a bounded subset of~$\mbR^2$, and~$M_i$ is the number of points 
that fall in ball~$i$.
\end{enumerate}
Here, we consider only sequences of \emph{sparse} occupancy models, in which, as the number of 
locations~$n$ increases, the expected number of occupants at each location remains bounded.

For multinomial occupancy models, \cite{Hwa08} prove an LCLT with optimal error rate~$\bigo(1/\sigma^2)$. 
Their argument relies on a simple observation: if the number of balls~$m$ were Poisson distributed rather than fixed, 
then the occupancy counts $\{M_i\}_{i \in [n]}$ would be independent. 
They prove the result by applying an LCLT for i.i.d.\ random variables to the Poissonized version of the problem and then use a de-Poissonization argument to transfer back to the original.
However, in more complex occupancy models such as \ER\ random graphs and germ--grain models, Poissonization does not 
lead to a similar simplification,
and it is therefore unclear how to adapt these arguments to such models.
See also \cite{McDonald2005} and references therein for other methods that have potential 
to prove LCLTs for sums of dependent variables, but which are not adapted to our applications.

A general method used to prove LCLTs, introduced in  \cite{Mcd80},
is to combine a (global) central limit theorem with some condition implying \emph{smoothness} of the distribution 
being approximated (here, $\law(\Wdeq)$). A common way of quantifying the smoothness, used in \cite{Mcd80} 
and~\cite{Penrose2011},  is to find an embedded sum of 
independent random variables which are themselves smooth, in the sense that they satisfy an LCLT. 
Here we use a different notion of the smoothness of a distribution, given in~\eq{SmoothDef} below, which is closely 
related to that given in \cite{Davis1995}, 
and elaborated on in \cite{Rol15}. The latter paper demonstrates that bounds in a global metric (such as the 
Kolmogorov metric) to a target distribution, combined with appropriate bounds on the smoothness term~\eq{SmoothDef}, 
imply an LCLT with error. For \ER\ random graphs, for example, if the (optimal) Kolmogorov bounds 
in \cite{Gol13} and the smoothness bounds implicit in Lemma~\ref{ERUP} of this paper are combined with the 
Landau--Kolmogorov 
inequality in \cite[Theorem 2.2]{Rol15}, then an LCLT is obtained, but with an error bound of 
order~$\bigo(\sigma^{-3/2})$, which is substantially worse than 
our target order of~$\bigo(1/\sigma^{2})$.

\cite{Bar19} (see also \cite{Rol05}) combine the smoothness approach with Stein's method for distributional 
approximation to establish a method for proving LCLTs with error, in settings where the dependence 
between summands can be described in terms of a \emph{Stein coupling} (see \cite{Che10}, as well as~\eqref{SteCouDef}
below).
The Stein coupling most commonly applied to the occupancy counts~$\Wdeq$ is a \emph{size biased coupling\/}; 
see Section \ref{sBSBC}. 
The standard method of constructing outcomes of a size biased version $\Wdeq^s$ of~$\Wdeq$ is to take the 
configuration of occupants, labelled $\mathcal{X}$, and to modify it so that a single location, chosen uniformly at random, 
now has~$d$ occupants, and the remaining configuration has the conditional distribution, given this event.
The number of locations~$\Wdeq^s$ with~$d$ occupants in this modified configuration then has the size biased distribution of $\Wdeq$.
Given such a construction, 
\cite{Bar19} demonstrate that to obtain an LCLT with error for $\Wdeq$ we must essentially do two things:
{\bf (i)}~establish a concentration inequality for 
\[
   \Psi \Def |\mbE\{\Wdeq^s-\Wdeq \giv \mathcal{X}\}-\mbE\{\Wdeq^s-\Wdeq\}|,
\] 
and {\bf (ii)} prove that the distribution of~$\Wdeq$ is smooth.

Establishing a concentration inequality for~$\Psi$ is often the hardest task.
In most cases,~$\Psi$ is a complicated expression (for example, see Lemma~\ref{GGL1}), to which it is unclear how 
to apply standard concentration results, including those related to Stein's method.
This significantly limits the scope of the bound in \cite{Bar19}.
Indeed, when the authors use this method to prove an LCLT with error rate $\bigo(\sigma^{-2} (\log\sigma)^{1/2})$ 
for the number of \emph{isolated} ($d=0$) vertices in an \ER\ random graph, they do so by demonstrating that, 
when $d=0$, $\Psi$ has a particularly simple expression, to which established concentration inequalities can be applied; however,
when $d>0$ , $\Psi$ is significantly more complex, and a new approach is required.
In this paper we demonstrate that the recent results in \cite{Bar18}, which provide a widely applicable method for deriving central moment inequalities, can often be used to establish the necessary concentration inequalities for these more complicated expressions.
To highlight the connection to \cite{Bar19}, in Theorem~\ref{BRRT} we rewrite their general bound in terms of central moment inequalities. By applying these two bounds in tandem we obtain a relatively general method for proving LCLTs with near optimal error rate.

When we consider germ--grain models, there is an additional complication: if we apply a standard size biased coupling, 
as described above, then it is difficult even to find an expression for~$\Psi$.
We overcome this issue by using the \emph{bounded size biased couplings\/} proposed in \cite{Bart18} for 
\be{
   W_d \Def n-\Wdeq \Eq \sum_{i=1}^n \I\{M_i \not= d\},
} 
the number of locations that do {\it not\/} have~$d$ occupants. There
the authors demonstrate how to construct size biased versions of~$W_d$ in occupancy models, by moving at most a 
single occupant from its original location.
Such  couplings allow us not only to find an expression for~$\Psi$ in germ--grain models, but also to improve 
the bound in the \ER\ random graph application.

Thus, in this paper, we piece together the general LCLT bounds in \cite{Bar19}, the central moment inequalities 
in \cite{Bar18}, the bounded size biased couplings in \cite{Bart18}, and the smoothness terms and bounds 
of \cite{Rol15}, to establish a robust method for proving LCLTs for the number of locations with~$d$ occupants 
in occupancy models, together with error bounds that are of the same order as would be expected for
sums of independent indicators, up to logarithmic factors. The logarithmic factors arise from the concentration inequalities used to handle $\Psi$, and can only be avoided using  our method by modification in special examples; see \cite[Remark~2.9]{Bar19}.
We emphasise that the contribution of the paper is twofold: obtaining LCLTs with good error rate in two non-trivial 
examples --- already an interesting and difficult undertaking --- and also providing a straightforward method 
for obtaining LCLTs with error that is applicable in a wide variety of other settings.

The paper is organised as follows. In Section \ref{Sec:Stein}, we provide the key result, Theorem~\ref{BRRT}, for  
establishing LCLTs with rate. In Section \ref{Sec:MainRes}, we state LCLTs with rate for occupancy models obtained 
by applying Theorem~\ref{BRRT}.  In Section~\ref{Sec:Proofs}, we apply the framework established in
 Section~\ref{Sec:Stein} to prove the results in Section~\ref{Sec:MainRes}. Section~\ref{sec:appenstn} gives a 
derivation of Theorem~\ref{BRRT} from the results of \cite{Bar19}, and Section~\ref{sec:branchbinom} contains two 
auxiliary results used in the proofs.

\section{Stein's method}\label{Sec:Stein}

Stein's method is a powerful tool for bounding the error in the approximation 
of a distribution of interest by another well--understood target distribution. It 
was first developed for approximation by the normal distribution in \cite{Stein1972, Stein1986},
and by the Poisson distribution in \cite{Chen1975}; see \cite{Barbour1992}, \cite{BarbourChen2005}, \cite{Chen2011} and \cite{Ross2011} for various introductions to the method.

\subsection{Stein's method for LCLTs}\label{SMLCLTs}

We use Stein's method to bound the distance between $W$ and an integer valued target distribution $Z$, in the 
\emph{total variation metric},
\[
   \dtv \bclr{\law(W),\law(Z)} \Def \sup_{A \subseteq \mathbb{Z}} | \mbP [W \in A] - \mbP[Z \in A] |,
\] 
as well as in a metric to capture the local differences,
\[
  \dloc \bclr{\law(W),\law(Z)} \Def \sup_{a \in \mathbb{Z}} | \mbP [W = a] - \mbP[Z = a] |.
\]

Following \cite{Rol07}, we use \emph{translated Poisson} distributions as the approximating family, instead of the 
discretized normal. 
We say that a random variable~$Z$ has the translated Poisson distribution, and write $Z \sim \TP(\mu,\sigma^2)$, 
if $Z-s \sim \Po(\sigma^2 + \gamma)$, where 
\ben{\label{TP-notation}
   s \Def \lfloor \mu - \sigma^2 \rfloor, \qquad \gamma \Def \mu - \sigma^2 - \lfloor \mu - \sigma^2 \rfloor,
}
and where $\Po(\lambda)$ denotes the Poisson distribution with mean~$\lambda$. 
Note that $\mbE Z = \mu$ and $\sigma^2 \leq \Var Z \leq \sigma^2 +1$. Thus, the translated Poisson distribution 
is a Poisson distribution translated by an integer so that both its mean and variance closely approximate $\mu$ 
and~$\sigma^2$.
The next result shows that the translated Poisson and the discretised normal are sufficiently close for the purposes of LCLTs. In particular, it implies that in the LCLTs given in Theorems~\ref{main} and~\ref{GGlclt} below, the translated Poisson distribution and the discretized normal distribution are interchangeable. The lemma follows by  applying~\eqref{iidPetrov}
to $(Z-s)\sim\Po(\sigma^2+\gamma)$ and using basic properties of the Poisson distribution and normal density.

\begin{lemma}
If $Z \sim \TP(\mu, \sigma^2)$, then there exists a constant $C >0$ such that, for all $\mu \in \mathbb{R}$ and 
$\sigma^2 \geq 1$, 
\[
  \sup_{n \in \mathbb{Z}} \left| \mbP(Z=n) - 
     \frac{1}{\sigma\sqrt{2\pi }}\exp \left( - \frac{(n-\mu)^2}{2 \sigma^2}\right) \right| \Le \frac{C}{\sigma^2}.
\]
\end{lemma}

The general LCLT theorem that we use requires that the variable of interest~$W$ is part of a \emph{Stein coupling\/}. 
Following \cite{Che10}, we say that the random variables $(W,W^{'},G,R)$ form an \emph{approximate Stein coupling\/} if 
\begin{equation}\label{SteCouDef}
    \mbE [ (W-\mu) f(W) ] \Eq \mbE [G( f(W^{'}) - f(W)) ]  + \mbE[R f(W)],
\end{equation}
for all $f$ such that the expectations exist. If $R=0$ almost surely, we call $(W,W',G)$ a \emph{Stein coupling\/}. 
The final ingredients needed for the general LCLT are the following standard probabilistic measures of 
smoothness for integer valued distributions:
\begin{equation}\label{SmoothDef}
  S_l(\law(W)) \Def \sup_{h\colon\norm{h}_{\infty}\leq 1} \abs{\mbE \Delta^l h(W)}, \qquad  l \geq 1,
\end{equation}
where $\Delta^l h(i) := \Delta^{l-1} (h(i+1) - h(i))$. Note that
$S_1(\law(W)) = 2 \dtv(\law(W), \law(W+1))$.

The following is a concise and more easily applicable modification of Corollary~2.3 and Lemma~2.6 of \cite{Bar19}, 
for proving LCLTs using Stein's method.  
As the statement here is not easily read from their results, we provide a proof below in Section~\ref{sec:appenstn}.
Here and below, for a random variable $X$ and $q>0$, we denote
\be{
\norm{X}_q:=\bclr{\mbE \bcls{\abs{X}^q} }^{1/q}.
}

\begin{theorem}\label{BRRT}
Let $(W,W',G,R)$ be an approximate Stein coupling with $W$ and~$W'$ integer valued, $\mbE W= \mu$ and $Var(W)= \sigma^2$. 
Set $D := W'-W$, and let $\mathcal{F}_1$ and $\mathcal{F}_2$ be sigma algebras such that~$W$ is $\mathcal{F}_1$-measurable 
and such that $(G,D)$ is $\mathcal{F}_2$-measurable. Define 
\begin{equation}\label{Updef}
   \Upsilon \Def \mbE\bcls{ \abs{ GD(D-1) }  \, S_2\bclr{\mathcal{L}(W| \mathcal{F}_2)}},
\end{equation}
and 
\[
T:= \bigl| \mbE[GD \giv \mathcal{F}_1] - \mbE[GD] \bigr|.
\]
If \ignore{$\sigma>1$ and}  $c_1$ is such that 
\ben{\label{eq:upsrti2bd}
      \max\Bigl\{\Upsilon+1,\norm{R}_2,\s^{-1}\norm{T}_2 \Bigr\} \Le c_1,
}
then
\ben{\label{eq:ezbd1}
  \dtv\left(\mathcal{L}(W), TP(\mu, \sigma^2) \right)  \Le 5c_1\sigma^{-1}.
}
If, in addition, $c_2$ is such that, for $q = \ceil{\log(\sigma)}$, 
\ben{\label{eq:ezbd2}
    \s^{-1}\norm{T}_q \Le c_2,
}
and if 
\ben{\label{c-condition}
     c_1 + ec_2 \ <\ \s/2,
}
then 
\ben{\label{eq:ezbd}
  \dloc\left(  \mathcal{L} (W), TP(\mu,\sigma^2) \right) \Le 4(4c_1 + ec_2)\s^{-2}.
}
\end{theorem}

\begin{remark}\label{c_2-dependence}
 In our applications, $c_1$ is fixed as~$\s$ increases, whereas $c_2 = c_2(\s)$ grows like
$(\log\s)^{\a}$, for some fixed $\a > 0$; hence the total variation bound is of order $O(\s^{-1})$,
whereas the local bound is of order $O(\s^{-2}(\log\s)^{\a})$.  
For such choices, \Ref{c-condition} is satisfied for all~$\s$ large enough.  If it were not satisfied,
then the bound in~\eqref{eq:ezbd} would be of order comparable to typical point probabilities, and
would thus be of little use.
We typically write $T=\sum_{l=1}^k T_l$, where the~$T_l$, $1\le l\le k$, have a structure enabling their norms
to be bounded using a main result of \cite{Bar18}, given as Theorem~\ref{Yuting} below,
for convenience and completeness. Then bounds related to~$T$ of~\eq{eq:upsrti2bd} and~\eq{eq:ezbd2} can be 
verified using the triangle (or Minkowski's) inequality for $\norm{\cdot}_q$ for $q\geq 1$.   The quantity~$R$ is zero almost surely.
It remains to bound~$\Upsilon+1$ by a constant~$c_1$.
 To do so, note that, in the typical regime,~$D$ is of constant order in~$\sigma$ and~$G$ is of 
order~$\sigma^2$.
Thus, establishing $\Upsilon+1\leq c_1$ boils down to showing that $S_2(\law(W)\giv \cF_2)=\bigo(\sigma^{-2})$. 
Conditioning on $\cF_2$ is handled by  ``freezing'' an asymptotically negligible part of the randomness, and then 
the term behaves similarly  to $S_2(\law(W))$. We can then apply established methods for bounding~$S_2(\cdot)$; 
see, for instance, \cite[Chapter II.14]{Lin02}, \cite{Mat07} and \cite{Rol15}.
\end{remark}

\begin{theorem}[Theorem~2.2 of \cite{Bar18}]\label{Yuting}
Suppose that $(Y,Y',G)$ is a Stein coupling  with $\mbE \{ |Y'-\mu |^r \} \leq \mbE \{ | Y-\mu |^r \}$, for some $r\in\IN$. 
Then 
\[
   \norm{Y- \mu}_r \Le \sqrt{2 (r-1) \norm{ G}_r \norm{Y'-Y}_r}.
\]
\end{theorem}

\subsection{Bounded size biased couplings}\label{sBSBC}

The Stein couplings that we use here, when we apply Theorem~\ref{BRRT}, are \emph{size biased couplings\/}. For a 
non-negative random variable~$W$, we say that $W^s$ has the $W$-size biased distribution if 
\begin{equation}\label{sizebiasdef}
    \mbE[W f(W) ] \Eq \mu \mbE [f (W^s) ] 
\end{equation}
for all functions $f$ for which these expectations exist; moreover, if $W^s$ and~$W$ are defined on the same probability 
space, then we say that $(W,W^s)$ form a \emph{size-biased coupling\/}. It is not difficult to verify that 
$(W,W',G)=( W, W^s, \mu)$ satisfies~\eqref{SteCouDef} with $R=0$, and is therefore a Stein coupling.

The standard method used to construct sized biased couplings for occupancy counts~$\Wdeq$, defined at \eqref{Occd}, 
such that $\IP(M_i=d)$ is independent of~$i$, relies on the identity:  
\begin{equation}\label{SBid}
   \mbE[\Wdeq f(\Wdeq)] \Eq \sum^n_{i=1} \mbP[M_i=d] \mbE\bigl(f(\Wdeq) \giv M_i=d\bigr) 
           \Eq \mbE[\Wdeq] \sum_{i=1}^n \frac{1}{n} \mbE \bigl( f(\Wdeq) \giv M_i =d \bigr).
\end{equation}
This identity gives the following standard method for coupling $\Wdeq^s$ to~$\Wdeq$, by conditioning on the occupants 
and locations: select a location $I\in [n]$ uniformly at random; then, if $M_I >d$, select $M_I -d$ occupants uniformly 
at random from those at location~$I$, and re-position them uniformly at random in $[n] \setminus I$; 
if $M_I <d$, select $d-M_I$ occupants uniformly at random from those at locations different from~$I$, and re-position 
them at location~$I$. The resulting number of locations with~$d$ occupants is then the corresponding outcome 
of~$\Wdeq^s$; for more details, see \cite[Section~2.3.4]{Chen2011}.

\cite{Bart18} show that it can be advantageous to instead construct a sized biased coupling 
for the number of locations that do {\it not\/} contain precisely $d$ occupants:
\[
     W_d \Def n-\Wdeq \Eq  \sum_{i = 1}^n \I \{M_i \neq d\}.
\]
As for~\eqref{SBid}, we have
\ben{\label{SBidne}
      \mbE[W_d f(W_d)] \Eq \mbE[W_d] \sum_{i=1}^n \frac{1}{n} \mbE \bigl( f(W_d)\giv M_i \not=d \bigr).
}
The strategy, as before, is to choose~$I$ uniformly at random.
Given $I=i$, modify~$M_i$, by adding or removing occupants of~$i$, in such a way that the distribution
of the number of occupants of~$i$ becomes $\mathcal{L}(M_i \giv M_i \neq d)$.  By
symmetry, if occupants added are chosen from the remainder uniformly at random, and if those removed are
re-distributed uniformly at random among the other locations, then the resulting distribution of the number
of locations with other than~$d$ occupants becomes $\mathcal{L}(W_d \giv M_i \neq d)$.
Of course, proving LCLTs for~$W_d$ is equivalent to doing so for~$\Wdeq$, but the advantage of working
with~$W_d$ is that {\it at most one\/} occupant has to be moved to realize the coupling above.
This is the substance of the following lemma, which is \cite[Lemma 2.1]{Bart18}.

\begin{lemma}[Lemma~2.1 of \cite{Bart18}]\label{Gold21}
If $M$ is an integer valued random variable such that $\law(M)$ is log--concave, and if
\begin{equation}\label{pigadef}
   \mbox{$\pi^{(d)}_x = \begin{cases}
            \frac{\mbP(M \geq x+1)\mbP(M=d)}{\mbP(M \geq d+1)\mbP(M=x)}, \quad & \mbox{if  } x \geq d \\
            0 , & \text{otherwise},
     \end{cases}$}  \qquad 
   \mbox{$\gamma^{(d)}_x = \begin{cases}
            \frac{\mbP(M \leq x-1)\mbP(M =d)}{\mbP(M \leq d-1)\mbP(M =x)}, \quad & if x \leq d \\
            0 , & \text{otherwise},
     \end{cases}$}
\end{equation}
then $\pi^{(d)}_x,\gamma^{(d)}_x \in [0,1]$ for all~$x$. Moreover, if $Z_+$ and $Z_-$ are conditionally independent 
given~$M$, with $\mathcal{L}(Z_+ \giv M) = \Be(\pi^{(k)}_{M})$ and $\mathcal{L}(Z_- \giv M) = \Be(\gamma^{(k)}_{M})$; 
if $Z$ is independent of $Z_+$ and~$Z_-$ with $\mathcal{L}(Z)=\Be(q)$, and 
\begin{equation}\label{qpdef}
     q \Def \frac{\mbP(M \geq d+1)}{\mbP(M \neq d)};
\end{equation}
and if
\[
   X \Def Z Z_+ - (1-Z)Z_-;
\]
then $\mathcal{L}(M+X) = \mathcal{L}(M \giv M \neq d)$.
\end{lemma}

\nin A distribution $\mathcal{L}(M)$ is log--concave when 
\begin{equation}\label{logconc}
    \mbP(M= s-1) \mbP(M = s+1)  \Le \mbP(M  = s)^2, \qquad \text{ for all integers }s;
\end{equation}
binomial distributions, in particular, are log concave.  The theorem shows that~$X$ occupants are to be
moved, and that $X \in \{-1,0,1\}$.

To understand the first assertion in Lemma~\ref{Gold21}, observe that 
$$
\gamma^{(d)}_x \leq 1 \quad \text{ if } \quad \frac{\mbP(M_i \leq y-1)}{\mbP(M_i =y)} 
                           \Le \frac{\mbP(M_i \leq d-1)}{\mbP(M_i = d)}
$$ 
for all $y \leq d-1$, which equivalent to 
$$
     \mbP(M_i \leq y-1)\mbP(M_i = y+1) \Le \mbP(M_i \leq y) \mbP(M_i =y).
$$ 
This, in turn, can be verified through repeated application of~\eqref{logconc}. A similar argument holds for~$\pi_x^{(d)}$.
To understand the second assertion of Lemma~\ref{Gold21}, observe that 
\begin{align*}
  \mbP(M_i + X = j) &\Eq  \mbP(M_i = j ) (1- q \pi^{(d)}_j -(1-q)\gamma^{(d)}_j) + \mbP(M_i = j-1) q \pi^{(d)}_{j-1} \\
   &\qquad \mbox{} + \mbP(M_i = j+1) (1-q) \gamma^{(d)}_{j-1}   \\
   &\Eq \mbP(M_i=j \giv M_i \neq d),
\end{align*}
as can be verified directly. We refer the reader to the proof of \cite[Theorem 2.1]{Bart18} 
for more details.

\ignore{
Lemma  \ref{Gold21} suggests a new method for generating outcomes of $W_d^s$ from a uniform occupancy configuration 
with counts $\{M_i \}_{i \in [n]}$: select a vertex $I$ uniformly at random; given $M_I$ simulate $X$ as suggested 
in Lemma \ref{Gold21}; finally select at most a single occupant uniformly at random to add or remove from location 
$I$ so that location $I$ now has $M_I + X$ occupants; the number of locations that do not contain precisely $d$ 
occupants is then the corresponding outcome of $W_d^s$. 
}

As pointed out in \cite{Bart18}, constructing $W_d^s$ in this manner leads to bounded size biased couplings, that is, 
couplings such that $|W_d^s - W_d| <C$ almost surely (here, for $C=2$). 
The advantage of this coupling is that we need only move at most a single occupant, and this makes it easier to 
express the quantity~$T$ arising in Theorem~\ref{BRRT}, particularly in Section~\ref{Sec:GGproof}. Moreover, 
using this coupling instead of the standard one reduces the additional logarithmic factors that occur when using Theorem~\ref{BRRT} to prove 
local limit theorems.

\section{Applications}\label{Sec:MainRes}

We present LCLTs with error in two applications:  the number of degree-$d$ vertices in an \ER\ random graph,
and the number of germs with $d$-neighbours in a germ--grain model.

\subsection{\ER\ random graphs}\label{Sec:ERres}

Let $\IG_n$ be the set of simple and undirected graphs with the vertex set $\mathcal{V}=\{v_1,\ldots,v_n\}$. 
We construct an \ER\ random graph $\cG_n$ on $\IG_n$ by letting the indicators $E_{ij}$, which determine the 
presence of an edge between $v_i$ and~$v_j$, be independent Bernoulli random variables with a common success 
probability~$p$ when $i \neq j$, and be zero when $i=j$. 
In this paper, we consider \emph{sparse\/} \ER\ random graphs, that is, we let $p = \lambda /n$ for some 
constant~$\lambda>0$.

We let $M_i = \sum_{j=1}^n E_{ij}$ be the degree of vertex $V_i$ in~$\cG_n$, and study the total number of 
vertices whose degree is not precisely~$d$:
\begin{equation}\label{ERwd}
    W_d \Def \sum_{i=1}^n \I \{ M_i \neq d \}.
\end{equation}
Using elementary arguments, we obtain 
\[
   \mu_d \Def \mbE W_d \Eq n(1-b_d)\qquad \mbox{and} \qquad 
    \sigma^2 \Def \Var(W_d) \Eq n b_d^2\left[\frac{(d-(n-1)p)^2}{(n-1)p(1-p)}-1\right]+n b_d,
\]
where $b_d:=\binom{n-1}{d}p^{d}(1-p)^{n-d-1}$. Observe that $\mu_d$ and~$\sigma^2$ are both of strict order~$n$.

We are now in a position to state an LCLT for $W_d$. This complements the work of a number of authors, who establish 
central limit theorems for~$W_d$ that apply when $p=\lambda/n$:
\cite{Bar89} prove that~$W_d$ is asymptotically normal and, in addition, obtain bounds in the Wasserstein metric of 
optimal order;
\cite{Kor90} obtains bounds of optimal order~$\bigo(\sigma^{-1})$ in the Kolmogorov metric when $d=0$, 
and \cite{Gol13} obtains bounds of optimal order in the Kolmogorov metric when $d \geq 0$; 
\cite{Fan14} obtains bounds between the distribution of~$W_d$ and an appropriately discretized normal distribution 
in the total variation metric that are of optimal order~$\bigo(\sigma^{-1})$; finally, \cite{Bar19} prove an 
LCLT with bounds of order~$\bigo(\sigma^{-2} \log(\sigma)^{1/2})$, but only for the case $d=0$.

\begin{theorem}\label{main}
For an \ER\ random graph with $p=\lambda/n$, if~$W_d$ is given by \eqref{ERwd} and $\sigma^2 := \Var(W_d)$, then
as $n\to\infty$, for any $d \geq 0$,
\ba{
  \dtv\bclr{\law(W_d), \TP(\mu_d,\sigma_d^2)}&\Eq \bigo\bclr{1/\sigma}; \\
  \dloc\bclr{\law(W_d), \TP(\mu_d,\sigma_d^2)}&\Eq \bigo\bbbclr{\frac{(\log\sigma)^{5/2}}{\sigma^2}}.
}
\end{theorem}

\subsection{Germ--grain models}\label{Sec:GGm}

To define the germ--grain models that we study, let $C_n:= [0, n^{1/2})^2$ be a torus; for $x,y \in C_n$, let $D(x,y)$ 
denote the distance between $x$ and $y$ under the Euclidean toroidal metric on $C_n$; and, for $x \in C_n$ 
and $s >0$, let $B_s(x)$ denote the ball $\{ y \in C_n: D(x,y) \leq s \}$. 
Let $V_1, \dots, V_n$ be independent points scattered uniformly in~$C_n$.
We refer to points in the set $\mathcal{V} :=\{ V_1, \dots, V_n\}$ as \emph{germs\/}. 
For a fixed value $r >0$, let $B_{i,r}:= B_r(V_i)$ be the $r$-ball that surrounds germ~$i$. We refer to 
these balls as \emph{grains\/}. To avoid small-$n$ boundary effects, we assume that $\pi r^2 < n$. Let 
\begin{equation}\label{SetIdef}
     \mathcal{I}_{i,r} \Def \left\{ j \neq i: V_j \in B_{i,r} \right\}
\end{equation}
be the set of germs that fall in grain~$i$, and $M_i:= \card\{\mathcal{I}_{i,r}\}$ be the number of germs that fall 
in grain~$i$. We study the total number of germs whose grain does not contain precisely~$d$ germs, that is,
\begin{equation}\label{Wggdef}
    W_d \Def \sum_{i=1}^n \I\{ M_i \neq d \}.
\end{equation}
Using \eqref{Wggdef} and the fact that $M_i \sim \text{Bi}(n-1, \pi r^2/n)$, we obtain $\mu_d:=\mbE(W_d)=n(1-b_d)$, 
where 
\[
   b_d \Def {n-1 \choose d} \left( \frac{\pi r^2}{n} \right)^d \left( 1 - \frac{\pi r^2}{n} \right)^{n-1-d}. 
\]
An expression for $\sigma^2:=\Var(W_d)$ is straightforward to derive, but more difficult to analyse asymptotically. 
For our purposes, it is enough to apply  \cite[Theorem~2.1]{Pen01}, which implies that $\sigma^2$ is also of strict 
order $n$. 

We are now in a position to state an LCLT for $W_d$. Our result complements the work of a number of authors:
\cite{Pen01} (see also \cite{Penrose2003}) establish general CLTs for geometric random graphs that apply to~$W_d$, which in that setting corresponds to the number of vertices not having degree~$d$;
\cite{Cha08} gives optimal bounds in the Wasserstein metric;
\cite{Gol10} obtain Berry--Essen bounds of optimal order $\bigo(\sigma^{-1})$ when $d=0$ and, when combined with 
the bounded size-biased couplings described in Section~\ref{sBSBC}, their method extends naturally to $d \geq 0$. 
\cite{Lac17} establish Berry--Essen bounds for functionals of binomial point processes; \cite[Section~6.1]{Penrose2011} 
give an LCLT without rate for~$W_0$ on the scale of its span, though it does not appear that it is established 
that the span is~$1$.
To the best of our knowledge, neither a bound in total variation nor an LCLT has previously been established for~$W_d$, 
when $d\geq 1$.

\begin{theorem}\label{GGlclt}
In the germ--grain model described above, for any fixed $r > 0$, if~$W_d$ is given by~\eqref{Wggdef} and 
$\sigma^2 := \Var(W_d)$, then as $n \to\infty$, for any $d \ge 0$, 
\begin{align*}
  \dtv\left( \mathcal{L}(W_d), TP(\mu, \sigma^2) \right) &= O(1/\sigma); \\
  \dloc \left( \mathcal{L}(W_d), TP(\mu, \sigma^2) \right) &= O\left( \frac{(\log\sigma)^{3/2}}{\sigma^2} \right).
\end{align*}
\end{theorem}

\section{Proofs of the applications}\label{Sec:Proofs}



We now prove Theorem \ref{main} (Section \ref{Sec:ERproof}) and Theorem~\ref{GGlclt} (Section~\ref{Sec:GGproof}).
In each section, we split the proof into lemmas; we note that, with the exception of Lemma~\ref{GGL3} 
(where the analogue in for \ER\ random graphs is trivial) the Lemmas in Sections \ref{Sec:ERproof} and~\ref{Sec:GGproof}
are in one-to-one correspondence.

\subsection{\ER\ random graphs}\label{Sec:ERproof}

To prove Theorem~\ref{main}, we first construct a size biased coupling.  To do so,
we define a new random graph~$\cG^s_n$ with vertex set $\mathcal{V}=\{ v_i \}_{i \in [n]}$ and edge indicators 
$\{E^s_{ij}\}_{i,j \in [n], i \neq j}$, in the following way.
We let $I$ be distributed uniformly on~$[n]$, independently of everything else. 
Given $I=i$ and~$M_i$, we let $Z$, $Z_+$, and $Z_-$ be independent Bernoulli random variables with means $q$, 
$\pi^{(d)}_{M_i}$ and~$\gamma^{(d)}_{M_i}$, where the expressions for $q$, $\pi^{(d)}_{M_i}$ and~$\gamma^{(d)}_{M_i}$
are as in Lemma~\ref{Gold21} ; we then let $X := Z Z_+ - (1-Z)Z_-$. 
If $I=i$ and $X=0$, we set $E^s_{lj}=E_{lj}$ for all $l,j\in [n]$; if $I=i$ and $X=1$, we let $J$ be uniformly distributed 
on $\{ j \in [n] : E_{ij}=0 \}$, conditionally independent of everything else, given $I=i$ 
and $(E_{ij},\,j \in [n]\setminus\{i\})$,
and we set $E^s_{iJ}=1$ and $E^s_{lj}=E_{lj}$ for all other pairs $l,j$; 
finally, if $I=i$ and $X=-1$, we let $J$ be uniformly distributed on $\{ j \in [n] : E_{ij}=1 \}$,
conditionally independent of everything else, given $I=i$ 
and $(E_{ij},\,j \in [n]\setminus\{i\})$, and we set $E^s_{iJ}=0$ and $E^s_{lj}=E_{lj}$ for all other pairs $l,j$. 
If we let~$W^s_d$ denote the number of vertices in~$\cG^s_n$ with degree different from~$d$, then~$W^s_d$ has the size 
biased distribution of~$W_d$, and $(W_d, W^s_d,\mu_d)$ is a Stein coupling.

The first lemma provides a useful expression for the quantity~$T$ that arises in Theorem~\ref{BRRT}, 
when we apply it to the Stein coupling $(W_d, W^s_d, \mu)$,
in terms of {\it local\/} statistics.
Fix $r \in \mbN$, and, for $G \in \IG_n$ and each $i =1, \dots, n$, let~$\mathcal{N}_r(i,G)$ be the 
`$r$-neighbourhood' consisting of the vertex--labelled subgraph induced by vertices at graph distance at most~$r$ from vertex~$i$.
Observe that~$M_i$ is a function of~$\mathcal{N}_1(i, \cG_n)$, and that
\[
   \hW_t(v_i) \Def \sum_{j=1}^n E_{ij}\I\{ M_j =t\},
\]
the number of degree~$t$ vertices connected to~$v_i$, is a function 
of~$\mathcal{N}_2(i,\cG_n)$; we refer to such statistics as local.
Lemma~\ref{ERT} shows that the quantity~$T$ that arises in Theorem~\ref{BRRT} can be bounded
in terms of a sum of centred sums of bounded local statistics.

\begin{lemma}\label{ERT}
We have
\[
    |\mbE[GD \giv \mathcal{G}_n]-\sigma^2| \Le (1-b_d)\sum_{l=1}^6 T_l,
\]
where $T_l = |T_l' - \ex T_l'|$, $1\le l\le 6$, and 
\begin{align*}
  T_1' &\Eq \frac{q}{n}\, U W_{d-1}; \qquad
  T_2' \Eq \frac{q}{n} U W_{d};\qquad T_6' \Eq W_d;\\
  T_3' &\Eq q  \sum_{i=1}^n \I [M_i < n/2]\, 
      \frac{ (\hW_{d-1}(v_i)- \hW_{d}(v_i)+\I[M_i=d-1]-\I[M_i=d]) \pi_{M_i}^{(d)}  }{n-M_i-1}\,; \\
  T_4' &\Eq q \sum_{i=1}^n \I[M_i  \geq n/2]\, 
           \frac{ \pi_{M_i}^{(d)}\{W_{d-1}+\hW_{d-1}(v_i)+\I[M_i=d-1] - W_d - \hW_d(v_i)-\I[M_i=d]\}}{n-M_i-1} \,; \\
  T_5' &\Eq (1-q)\sum_{i=1}^n \,\frac{ (\hW_{d+1}(v_i)-\hW_d(v_i)) \gamma_{M_i}^{(d)}  }{M_i}\,,  
\end{align*}
where
\[
     U \Def \sum_{i=1}^n \I [M_i < n/2]\, \frac{n  \pi_{M_i}^{(d)}  }{n-M_i-1}\,.
\]
\end{lemma}

\begin{proof}
By considering the degree of the vertex~$I$ chosen and which of its neighbours gain or lose an edge, 
we obtain
\begin{align*}
  \mbE[ GD \giv \cG_n] &\Eq  (1-b_d) \sum_{i=1}^n \biggl(  \I\{M_i =d \} 
                 + \frac{(1-q) \gamma_{M_i}^{(d)}}{M_i}\, [\hW_{d}(v_i)- \hW_{d+1}(v_i) ] \phantom{XXXXX}\\
     &\qquad\quad\mbox{} + \frac{q\pi_{M_i}^{(d)}}{n-M_i-1}\, \biggl[-(n-W_{d-1}) + \hW_{d-1}(v_i)+\I[M_i=d-1] \\ 
     &\qquad\qquad\qquad\qquad\qquad\qquad\mbox{}        + (n- W_d) - \hW_d(v_i)-\I[M_i=d]\biggr] \biggr);
\end{align*}
note that, if $M_i = 0$ or $M_i=n-1$, so that the denominator in one of the fractions is zero, the numerator
is zero also, and the corresponding term is to be taken as zero.
The lemma then follows by observing that $\sigma^2=\mbE[GD]$, by rearranging the terms, and by applying the 
triangle inequality. 
\end{proof}

Note that $T_3'$, $T_5'$, $W_d$ and~$W_{d-1}$ are sums of local statistics that are bounded by~$1$, that~$U$
is a sum of local statistics bounded by~$4$, if $n\ge2$, and that~$T_4'$ is at most~$n$ times
the number of~$i$ such that $M_i > n/2$, whose expectation is very small.  These observations help in controlling
the norms of $\s^{-1}T_i$, $1\le i\le 6$, as required when applying Theorem~\ref{BRRT}, because
Theorem~\ref{Yuting} can be invoked.  For this example, the result of doing so has already been established as
 \cite[Theorem 4.2]{Bar18}, and is therefore stated here without proof.

\begin{lemma}\label{ERwen}
Let $\mathrm{U}$ be a real valued function on all vertex--labelled graphs with one distinguished vertex~$v$, and at most $n-1$ other vertices.
Suppose that there exists constants~$c$ and $\beta\geq 0$ such that 
\[
    |\mathrm{U}(G,v)| \Le c \,\card\{V(G)\}^\beta,
\]
for every $G$ in the domain of $\mathrm{U}$.
Fix $\lambda >0$ and let $\cG_n$ be an \ER\ random graph on $\IG_n$, with edge probability $p := \lambda/n$; define 
\[
   X_i \Def \mathrm{U}\bclr{\mathcal{N}_r(i,\cG_n),i}
\]
and $W= \sum_{i=1}^n X_i$. Then, for any $q \ge 1$, 
\[
  n^{-1/2} \norm{W- \mbE W}_q \Le c K(\b)[C_A \max\{\lambda, q(1+\beta) \}]^{1/2+r +2r\beta},
\]
where  $K(\b) := \sqrt{2}(10^{1+\beta} + 2^{1+\beta})$ and $C_A := \pi e^{e-2}/\log(e-1)$.
\end{lemma}

The implications of Lemma~\ref{ERwen} for the moments of~$\s^{-1}T_l$, $1\le l\le 6$, are as follows.

\begin{corollary}\label{ERmoments}
 For each $1 \le l\le 6$ and each $1 \le q \le \lceil\log\s\rceil$, we have $\s^{-1}\norm{T_l}_q \le bq^{5/2}$,
for a suitable fixed choice of~$b$.
\end{corollary}

\begin{proof}
  Taking $c=1$, $\b=0$ and $r=2$, it is immediate from Lemma~\ref{ERwen} that 
$$
      n^{-1/2}\norm{T_i}_q \Le K(0) [C_A \max\{\lambda, q \}]^{5/2} , \qquad i = 3,5,
$$
and, since~$\s^2$ is strictly of order~$n$, it follows that, for fixed $b^{(3)}$ and~$b^{(5)}$,
\besn{\label{T35-norm}
   \s^{-1} \norm{T_i}_q &\Le b\uii q^{5/2}, \quad i=3,5, \quad\mbox{for all}\ q \ge 1. 
} 
For $W_d$, $W_{d-1}$ and~$U$, now taking $r=1$, it follows similarly that
\[
    n^{-1/2}\max\{\norm{W_d}_q,\norm{W_{d-1}}_q,\norm{U/4}_q\} \Le b^{(0)} q^{3/2},
\]
for all $q \ge 1$, for some fixed~$b^{(0)}$.  To convert these bounds into bounds on the moments of~$T_1$ and~$T_2$,
we use the following inequality.
Let $X$ and~$Y$ be integrable random variables with means $\m_X$ and~$\m_Y$, and write $X' := X - \m_X$ and $Y' := Y - \m_Y$;
then
\[
    XY - \mbE\{XY\} \Eq X'Y' + \m_X Y' + \m_Y X' - \mbE\{X'Y'\}.
\]
Hence, using the triangle inequality and Cauchy--Schwarz, 
\besn{\label{CSXY}
  \norm{XY-& \mbE\{XY\}}_q \\
  	&\Le \norm{X'Y'}_q + |\m_X|\norm{Y'}_q + |\m_Y|\norm{X'}_q + \mbE|X'Y'| \\
                &\Le 2 \norm{X-\mbE X}_{2q} \norm{Y-\mbE Y}_{2q} +  |\mbE X| \norm{Y-\mbE Y}_q + |\mbE Y| \norm{X-\mbE X}_q. 
}
Taking $X = n^{-1/2}W_{d-1}$ and $Y = n^{-1/2}U$ thus gives
\bes{
  \norm{T_1}_q &\Le 8 \{b^{(0)}\}^2 (2q)^{3} + 8n^{1/2}b^{(0)} q^{3/2}; 
}
hence, for $q \le \lceil \log\s \rceil$,
\ben{\label{T1-norm}
     \s^{-1}\norm{T_1}_q \Le b^{(1)} q^{3/2},
}
for some fixed~$b^{(1)}$, since $\s^{-1}\lceil \log\s \rceil^{3/2}$ is bounded in $\s \ge 1$.  The same argument
works also for~$T_2$.

Finally, for~$T_4$, note that~$\sum_{i=1}^n I[M_i > n/2]$ has maximum value~$n$, and that the probability that it
does not take the value~$0$ is bounded by~$\eps_n := nCe^{-n/6}$, for a suitable constant $C = C(\l)$
and for $n > 4\l$, in view of a simplified Chernoff inequality read from, for example, \cite[Theorem~2.3(b)]{McD98}:  for a sum~$S$ of independent Bernoulli random variables 
with mean~$\m$,
\ben{\label{Chernoff}
    \pr[S > (1+\d)\m] \Le e^{-\d\m/3}, \qquad \mbox{if}\ \d \ge 1.
}  
Hence it follows that, for a suitably chosen~$b^{(4)}$,
\[
    \s^{-1}\norm{T_4}_q \Le \s^{-1} n^{2} \norm{\Be(\eps_n)}_q \Eq \s^{-1} n^2 \eps_n^{1/q} \Le b^{(4)},
\]
for all~$1 \le q \le \lceil\log\s\rceil$.
Combining this with \Ref{T35-norm} and~\Ref{T1-norm} completes the proof of the corollary.
\end{proof}

\nin
Hence, in particular, when applying Theorem~\ref{BRRT}, we can take $c_2 = 6 b\ceil{\log\s}^{5/2}$ in~\eqref{eq:ezbd2}, 
and $\s^{-1}\norm{T}_2 \le 6 \times 2^{5/2}b$ in~\eqref{eq:upsrti2bd}.

All that is now needed, in order to apply Theorem~\ref{BRRT}, is a
bound of order~$\bigo(1)$ for the smoothness term~$\Upsilon$, to be used in \Ref{eq:upsrti2bd}. 
We derive it by applying \cite[Theorem 3.7]{Rol15}, using arguments that are based on those 
in \cite[Section 2.3.1]{Fan14}. In what follows, we define the index sets 
$A_I :=\{ I \} \cup \{ j \colon E_{Ij}=1\} \cup \{J\}$ and 
$B_I=\{  j \notin A_I \colon E_{jk}=1 \text{ for some } k \in A_I \}$. We then set 
\[
     \mathcal{F}_2 \Def \sigma(I, A_I, B_I, J, X),
\]
observing that $D$ (and $G$) are $\mathcal{F}_2$ measurable. 

\begin{lemma}\label{ERUP}
    For $\Upsilon = \mbE[ |GD(D-1)| S_2( \mathcal{L}(W_d \giv \mathcal{F}_2))]$, we have $\Upsilon=\bigo(1)$.
\end{lemma}

\begin{proof}
Because the size biased configuration is formed by altering at most a single edge of $\mathcal{G}_n$,
we have $|D|\leq 2$, and hence $|D(D-1)| \le 6$. Thus
\begin{align}
  \Upsilon &\Eq \mbE[ |GD(D-1)| S_2( \mathcal{L}(W_d \giv \mathcal{F}_2))] 
        \Le 6 n (1-b_d) \mbE[ S_2( \mathcal{L}(W_d \giv \mathcal{F}_2))] \nonumber \\
   &\Le  6 n (1-b_d) \mbE\left( S_2( \mathcal{L}(W_d \giv \mathcal{F}_2)) \I[\max\{|A_I| , | B_I |\} \leq \sqrt n]\right) 
        \non  \\ &\qquad\qquad\qquad\qquad\qquad\quad   + 6 n(1- b_d) \mbP[ \max\{|A_I|,|B_I|\} > \sqrt n]. \label{UP1}
\end{align}
To bound the second term in \eqref{UP1}, first observe that the distribution of~$|A_I|$ is stochastically dominated by 
$\Bi(n-1,p)+2$. Thus, by~\Ref{Chernoff}, there exists $C=C(\l)$ such that
\begin{equation}\label{McdA}
    \mbP[|A_I | > \sqrt{n}] \Le Ce^{-\sqrt{n}/3}
\end{equation}
whenever $\sqrt n \ge 2(\l+1)$.
Next observe that if $Y_1, Y_2, \dots$ are a sequence of i.i.d.\ $\Bi(n-1,p)$--distributed random variables, 
then, using the standard exploration process coupling, $|B_I|$ 
is stochastically dominated by $\sum_{i=1}^{|A_I|} Y_i$. Using~\Ref{Chernoff} again in \eqref{McDB1}, 
we then have 
\begin{align}
  \mbP\left( |B_I| > \sqrt n\right) &\Le \mbP(|A_I| > n^{1/4}) + \mbP(|B_I| > \sqrt{n} \giv |A_I| \leq {n}^{1/4}) \non \\
  &\Le \mbP(|A_I| > n^{1/4})+ \mbP( \exists \;  i\in\{1, \dots, \lfloor n^{1/4} \rfloor \} \colon\, Y_i > n^{1/4}) \non \\
  &\Le  (1+ n^{1/4}) C'e^{-n^{1/4}/3}\label{McDB1}\\
  &=\bigo(n^{-1}).\label{McDB2}
\end{align}
Combining \eqref{McdA} and \eqref{McDB2}, we see that the second term in \eqref{UP1} is of order~$\bigo(1)$.

To bound the first term in \eqref{UP1}, given $\mathcal{F}_2$, we define a random graph $\mathcal{G}_n^{\mathcal{F}_2}$ 
with vertex set $\mathcal{V}$ and edge indicators $E^{\mathcal{F}_2}_{uv}$ by letting $E^{\mathcal{F}_2}_{ij} = E_{ij}$ 
for $i \in A_I$ and $j \in \{ 1, \dots, n \}$, and letting $E^{\mathcal{F}_2}_{ij}$ be independent~$\Be(p)$ random 
variables for $i,j\in (A_I)^c$. If we let $W_d^{\mathcal{F}_2}$ denote the number of vertices in 
$\mathcal{G}_n^{\mathcal{F}_2}$ with degree different from~$d$, then 
$\mathcal{L}( W_d^{\mathcal{F}_2}) = \mathcal{L}(W_d \giv {\mathcal{F}_2})$.
We now show that, for any fixed $\mathcal{F}_2$ with $\max\{| A_I |,| B_I |\} \leq \sqrt{n}$, we have
\begin{equation}\label{P2}
     S_2( \mathcal{L}(W_d \giv \mathcal{F}_2)) \Eq \bigo(n^{-1}),
\end{equation}
by applying \cite[Theorem 3.7]{Rol15}.
For ease of notation, in the remainder of the proof, we suppress the superscript~$\mathcal{F}_2$, tacitly assuming that 
every random quantity has distribution conditional on~$\mathcal{F}_2$.

Let $\mathcal{G}_n$ be as above. Let~$\mathcal{G}'_n$ be the graph obtained by choosing a pair of distinct 
vertices $v_i$ and $v_j$ with
\[ 
   \{i,j\}\ \subset\ C_I \Def (A_I \cup B_I)^c,
\] 
uniformly at random, and resampling the edge indicator between $v_i$ and~$v_j$. Let $\mathcal{G}''_n$ be the 
graph obtained by applying the same 
operation to $\mathcal{G}'_n$. If we let $\widetilde W_d$, $\widetilde W'_d$ and $\widetilde W''_d$ be the numbers of 
vertices with degree different from~$d$ in $\mathcal{G}_n$, $\mathcal{G}'_n$ and~$\mathcal{G}''_n$, respectively, 
then $(\widetilde W_d, \widetilde W'_d ,\widetilde W''_d)$ are three successive states of a reversible Markov chain. 
Thus, if 
\[ 
     Q_{\pm 1}(G) \Def \mbP\bigl[ \widetilde W'_d = \widetilde W_d \pm 1 \giv \mathcal{G}_n = G\bigr]
\] 
and
\[
    Q_{\pm 1, \pm 1}(G) \Def 
  \mbP\bigl[\widetilde W'_d = \widetilde W_d \pm 1 , \widetilde W''_d = \widetilde W'_d \pm 1  \giv \mathcal{G}_n =G \bigr],
\]  
then, by \cite[Theorem 3.7]{Rol15}, we have
\begin{align}
   \label{RossEq}
\begin{split}
     S_2(\mathcal{L}(\tilde W_d \giv \mathcal{F}_2 )) 
        &\Le \frac{1}{(\mbE(Q_1(\mathcal{G}_n)))^2} \Bigl[ 2 \text{Var} Q_{1}(\mathcal{G}_n) + 
          \mbE |Q_{1,1}(\mathcal{G}_n)-Q_1(\mathcal{G}_n)^2| \\
     &\qquad\qquad\quad\mbox{} + 2 \text{Var} Q_{-1}(\mathcal{G}_n) + \mbE| Q_{-1,-1}(\mathcal{G}_n) - Q_{-1}(\mathcal{G}_n)^2| \Bigr].
\end{split}
\end{align} 
The remaining argument shows that this quantity is of order~$\bigo(n^{-1})$.

We first need expressions for $Q_{\pm 1}(G)$.  For $s,t \geq 0$, we let ${H}_{s,t}^{C_I}(G)$ ($R_{s,t}^{C_I}(G))$ 
denote the numbers of \emph{connected} (\emph{disconnected}) vertex pairs $\{v_i,v_j\}$ in~$G$ 
that have degrees $s$ and~$t$, and are such that $i,j \in C_I$. We also let ${H}_{s}^{C_I}(G)=\sum_{t \geq 0} H_{s,t}^{C_I}(G)$ 
($R_{s}^{C_I}(G)=\sum_{t \geq 0} R_{s,t}^{C_I}(G)$) be the numbers of connected (disconnected) vertex pairs $\{v_i,v_j\}$ 
in~$G$ such that at least one of $v_i,v_j$ has degree~$s$, and such that $i,j \in C_I$. We now have
\begin{align}
    Q_1(G) &\Eq p\frac{R_d^{C_I}(G) - R_{d,d}^{C_I}(G)-R_{d-1,d}^{C_I}(G)}{{|C_I| \choose 2}}
            +(1-p)\frac{H_d^{C_I}(G) - H_{d,d}^{C_I}(G)-H_{d+1,d}^{C_I}(G)}{{|C_I| \choose 2}} \,;\label{Qp1Eq}\\
    Q_{-1}(G) &\Eq p\frac{R_{d-1}^{C_I}(G) - R_{d-1,d}^{C_I}(G) - R_{d-1,d-1}^{C_I}(G)}{{|C_I| \choose 2}} \non\\
          &\qquad\qquad\mbox{}   +(1-p)\frac{H_{d+1}^{C_I}(G) - H_{d,d+1}^{C_I}(G) - H_{d+1,d+1}^{C_I}(G)}{{|C_I| \choose 2}}\,. \label{Qm1Eq}
\end{align}
These equalities are obtained through elementary considerations. For example, the first term in \eqref{Qp1Eq} is obtained 
by observing that the number of degree $d$ vertices decreases by exactly~$1$ when the chosen vertices are connected 
after resampling, and were previously disconnected, with one having had degree~$d$ 
and the other having had neither degree $d$ nor $d-1$; if both had had degree~$d$, 
$\widetilde W'_d$ would have exceeded $\widetilde W_d$ by~$2$, and if one had degree~$d$ and the other degree~$d-1$ 
there would be no change.

To prove that $\mbE(Q_{\pm 1}(\cG_n))$ is of strict order~$n^{-1}$, so that the denominator in~\Ref{RossEq} results
in a factor of order $\bigo(n^2)$,  we observe that, conditional on $\mathcal{F}_2$, 
the degree of each vertex~$v_i$ with $i\in C_I$ is distributed as $\Bi(|B_I|+|C_I|-1,p)$. Thus, if we let 
{\small
\[
    b_s^{\mathcal{F}_2} \Def {|B_I|+|C_I|-2 \choose s} p^s (1-p)^{|B_I|+|C_I|-2-s},
\]}
then we have   
\begin{align}
   \label{EQpm1Eqs}
\begin{split}
   \mbE(R^{C_I}_{s,t}(\cG_n)) &\Eq {|C_I| \choose 2} 
               (1-p)[2 b_s^{\mathcal{F}_2}b_t^{\mathcal{F}_2} \I \{s \neq t\} + (b_s^{\mathcal{F}_2})^2 \I\{s=t\}], \\
   \mbE(R^{C_I}_{s}(\cG_n)) &\Eq {|C_I| \choose 2}  
               (1-p)[2b_s^{\mathcal{F}_2}-(b_s^{\mathcal{F}_2})^2], \\
   \mbE( H^{C_I}_{s,t}(\cG_n)) &\Eq {|C_I| \choose 2} 
          p[2 b_{s-1}^{\mathcal{F}_2}b_{t-1}^{\mathcal{F}_2} \I \{s \neq t\} + (b_{s-1}^{\mathcal{F}_2})^2 \I\{s=t\}], \\
   \mbE( H^{C_I}_{s}(\cG_n)) &\Eq {|C_I| \choose 2} p [2b_{s-1}^{\mathcal{F}_2}-(b_{s-1}^{\mathcal{F}_2})^2].
\end{split}
\end{align}
Combining \eqref{EQpm1Eqs} with \eqref{Qp1Eq} and \eqref{Qm1Eq}, we obtain
\begin{align}
\label{EQpm1F}
\begin{split}
    \mbE\{Q_1(\cG_n)\} &\Eq 2p(1-p)[b_d^{\mathcal{F}_2}+b_{d-1}^{\mathcal{F}_2} -(b_d^{\mathcal{F}_2})^2
                -(b_{d-1}^{\mathcal{F}_2})^2-2b_d^{\mathcal{F}_2}b_{d-1}^{\mathcal{F}_2}]\\
          &\Eq 2p(1-p)( b_d^{\mathcal{F}_2}+b_{d-1}^{\mathcal{F}_2} )(1-b_d^{\mathcal{F}_2}-b_{d-1}^{\mathcal{F}_2}) \\    
     &\Eq \mbE(Q_{-1}(\cG_n)),
\end{split}
\end{align}
which, for fixed~$d$ and with $p = \l/n$ for fixed~$\l$, are both of strict order $n^{-1}$.

To bound $\mbE| Q_{1,1}(\mathcal{G}_n)-Q_1(\mathcal{G}_n)^2|$, for $G \in \IG_n$ and $C_I \subset \mathcal{V}$, 
we let $\IG_n^{(1)}(G)$ be the set of graphs in $\IG_n$ that can be obtained by modifying at most a single edge in~$G$ 
whose ends are both in $C_I$. By observing that modifying a single edge can alter the degrees of at most two vertices,
 we obtain, for $G \in \IG_n$, $G^{'} \in \IG_n^{(1)}(G)$, and $s,t \geq 0$,
\begin{align} 
  \label{BQpm1Eqs}
 \begin{split}
         &|R^{C_I}_{s}(G)- R^{C_I}_{s}(G^{'})| \Le 2|C_I|, \qquad |H^{C_I}_{s}(G)-H^{C_I}_{s}(G^{'})| \Le 2s, \\
         &| R^{C_I}_{s,t}(G) - R^{C_I}_{s,t}(G^{'})| \Le 2|C_I|, \qquad |H^{C_I}_{s,t}(G)-H^{C_I}_{s,t}(G^{'})|
               \Le 2\max\{s,t\}.
 \end{split}
\end{align}
Combining \eqref{BQpm1Eqs} with \eqref{Qp1Eq} and \eqref{Qm1Eq} we have
\begin{align}\label{Qpm1nb}
     |Q_{\pm 1}(G)-Q_{\pm 1}(G^{'})| \Le \frac{6p |C_I| + 6(1-p)(d+1)}{{|C_I| \choose 2}},
\end{align}
for any $G \in \IG_n$, $G^{'} \in \IG_n^{(1)}(G)$, and $C_I \subseteq \mathcal{V}$.
Using \eqref{EQpm1F} and \eqref{Qpm1nb} in \eqref{Q2sb1} below, and the prescription that  
$p = \l/n$ and $\max\{|A_I|,|B_I|\} < \sqrt{n}$ in \eqref{Q2sb2},
we obtain
\begin{align}
   \mbE| &Q_{1,1}(\mathcal{G}_n)-Q_1(\mathcal{G}_n)^2| \nonumber \\
      &= \sum_{G \in \IG_n}\mbP(\cG_n =G) |Q_{1,1}(G)-Q_1(G)^2| \nonumber \\
      &= \sum_{G \in \IG_n} Q_1(G)\mbP(\cG_n =G) \left| \left( \sum_{G^{'} \in \IG_n^{(1)}(G)} Q_{1}(G^{'}) 
              \mbP(\cG_n^{'} = G^{'} | \cG_n=G) \right)-Q_1(G) \right| \nonumber \\
      &\leq \sup_{G \in \IG_n, \, G^{'} \in \IG_n^{(1)}(G)}\left\{| Q_1(G^{'})-Q_1(G) |\right\} \mbE (Q_1(\cG_n)) \non \\
      &\leq 12\frac{p |C_I| + (1-p) (d+1)}{{|C_I| \choose 2}}\, 
           p(1-p)( b_d^{\mathcal{F}_2}+b_{d-1}^{\mathcal{F}_2} )(1-b_d^{\mathcal{F}_2}-b_{d-1}^{\mathcal{F}_2})
                \label{Q2sb1} \\
      &=\bigo(n^{-3}); \label{Q2sb2}
\end{align}
and the same arguments apply to $\mbE| Q_{-1,-1}(\mathcal{G}_n)-Q_{-1}(\mathcal{G}_n)^2|$. Since also 
\begin{equation}\label{Varn3}
      \Var(Q_{\pm 1}(\cG_n)) \Eq \bigo(n^{-3}),
\end{equation}
in view of \cite[Pages 1416--1418]{Fan14},
combining \eqref{EQpm1F}, \eqref{Q2sb2} and  \eqref{Varn3} with \eqref{RossEq} gives~\eqref{P2}, and hence the result.
\end{proof}

\ignore{
\begin{lemma}\label{ERS}
We have $S_2(W_d)=O(n^{-1})$.
\end{lemma}
\begin{proof}
The result follows from the same arguments used to establish \eqref{P2} but in this case there is no need to 
condition on $\mathcal{F}_2$.
\end{proof}
}

Theorem~\ref{main} now follows directly from Theorem~\ref{BRRT}, in view of the bounds on the quantities appearing
in \eqref{eq:ezbd2} and~\eqref{eq:upsrti2bd} established in Corollary~\ref{ERmoments} and Lemma~\ref{ERUP}.

\subsection{Germ--grain models}\label{Sec:GGproof}

To prove Theorem \ref{GGlclt}, we construct a size biased coupling, based on a new configuration  
 $\mathcal{V}^s = \{ V_1^s , \dots, V_n^s \}$.
Let  $\pi^{(d)}_{M_I}$, $\gamma^{(d)}_{M_I}$ and~$q$ be as given in \eqref{pigadef} and~\eqref{qpdef}.
Let $Z \sim \Be(q)$ and~$I$, distributed uniformly on~$[n]$, be independent of everything else.
Given~$M_I$, let $Z_+ \sim \Be(\pi^{(d)}_{M_I})$ and $Z_- \sim \Be ( \gamma^{(d)}_{M_I})$ be conditionally independent. 
Set 
\begin{equation}\label{Xdefi}
   X \Def Z Z_+ - (1-Z)Z_-.
\end{equation}
If $X=0$, let $\mathcal{V}^s = \mathcal{V}$. If $X=-1$, let $J$ be distributed uniformly on $\mathcal{I}_{I,r}$, 
defined in \eqref{SetIdef}, independently of everything else; then let $V^s_J$ be distributed uniformly 
on $C_n \setminus B_{I,r}$, independently of everything else, and set $V_i^s =V_i$ for all $i \neq J$. 
If $X=1$, let $J$ be distributed uniformly on (noting the modified set definition)
$$
   \hcI^c_{I,r} \Def [n] \setminus (\mathcal{I}_{I,r} \cup \{I\}),
$$
independently of everything else; then let $V^s_J$ be  distributed uniformly on $B_{I,r}$, and set $V_i^s =V_i$ for all 
$i \neq J$. By considering~\eq{SBidne}, if we let $W^s_d$ be the number of germs in $\mathcal{V}^s$ whose $r$-grain does not contain exactly $d$ germs, then $W^s_d$ has the size biased distribution of $W_d$; see also \cite[Section~3.2.2]{Bart18}.

Given this coupling, the proof of Theorem \ref{GGlclt} is split into several lemmas. 
To state them, we first require some definitions. For $s>0$ and $x \in C_n$, let 
\begin{equation}\label{Isx}
    \mathcal{I}_{s}(x) \Def \left\{ j \in [n]\colon\, V_j \in \mathcal{V} \cap B_{s}(x) \right\}
\end{equation}
be the set of germs contained in $B_{s}(x)$, and write
\begin{equation}\label{Ndef}
    N_s(x) \Def \card\{\mathcal{I}_{s}(x)\}
\end{equation}
be the number of germs contained in $B_{s}(x)$.
Given $\mathcal{V}$, for $i\in [n]$ and $j \in \hcI^c_{i,r}$, let $A_{ij}$ be the expected increment in~$W_d$ when~$V_j$ 
is moved to a uniformly selected location in $B_{i,r}$ and, for $i \in [n]$ and $j\in \mathcal{I}_{i,r}$, let~$R_{ij}$ be 
the expected increment in~$W_d$ when~$V_j$ is moved to a uniformly selected location in $C_n \setminus B_{i,r}$. 
To help express $A_{ij}$ and~$R_{ij}$, let
\begin{equation}\label{ExpS}
\begin{split}
   S_j-1 &\Def -\I \{M_j \not=d \}+ \sum_{\ell \in \mathcal{I}_{j,r}} [\I\{M_\ell=d\} - \I \{M_\ell = d+1\}] 
\end{split}
\end{equation}
be the increment in~$W_d$ when~$V_j$ is removed from~$\mathcal{V}$, and
\begin{equation}\label{ExpH}
  H_i +1 \Def \frac{1}{\pi r^2} \int_{B_{i,r}} dx \biggl\{ \I\{ N_r(x)\not=d \} + 
           \sum_{\ell \in \mathcal{I}_{r}(x)} [\I\{M_\ell = d\}-\I\{M_\ell=d-1\}]  \biggr\}
\end{equation} 
be the expected increment in~$W_d$ when an additional germ is inserted uniformly at random in~$B_{i,r}$.
At first, it may appear that $A_{ij} = S_j+H_i$; however, if removing~$V_j$ from~$\mathcal{V}$ causes the value of~$H_i$ 
to change, then this is not the case. We let 
\begin{align}
   Q_{ij}&\Def \frac{1}{\pi r^2} \int_{B_{i,r}}  dx 
              \Bigl\{ \I\{x \in B_{j,r}\}[-\I\{ N_r(x) \not=d\} + \I\{ N_r(x) \not=d+1\}] \nonumber \\
    &\qquad\mbox{} +   \sum_{\ell \in  \mathcal{I}_{r}(x) \cap \mathcal{I}_{j,r}} 
               [\I\{M_\ell=d+1\}-2\I\{M_\ell=d\}+\I\{M_\ell=d-1\}] \Bigr\} \notag \\
  \begin{split}
    &\Eq \frac{1}{\pi r^2} \int_{B_{i,r}}  dx \Bigl\{ \I\{x \in B_{j,r}\}[\I\{ N_r(x) =d\} - \I\{ N_r(x) =d+1\}] \\
    &\qquad\mbox{} +   \sum_{\ell \in  \mathcal{I}_{r}(x) \cap \mathcal{I}_{j,r}} 
                [\I\{M_\ell=d+1\}-2\I\{M_\ell=d\}+\I\{M_\ell=d-1\}] \Bigr\}  
  \end{split}\label{ExpQ}
\end{align}
be the increment in~$H_i$ caused by removing $V_j$ from~$\mathcal{V}$.  Observe that 
\begin{equation}\label{QFR}
    Q_{ij} \Eq 0 \quad\mbox{if}\quad D(V_i, V_j)\ >\ 3r. 
\end{equation}
With these definitions, we now have 
\begin{equation}\label{AExp}
     A_{ij} \Eq H_i +S_j + Q_{ij}.
\end{equation}

To express $R_{ij}$, first note that, when a germ is inserted uniformly into~$B_{i,r}^c$, then the expected change 
in~$W_d$ is given by 
\begin{align}
   K_i +1&\Def  \frac{1}{n-\pi r^2} \int_{B_{i,r}^c} dx \Bigl\{ \I\{ N_r(x)\not=d \} 
              + \sum_{\ell \in \mathcal{I}_{r}(x)}^n [\I\{ M_{\ell}=d\} - \I\{ M_{\ell}=d-1 \}] \Bigr\}\nonumber \\
     &\Eq \frac{1}{n-\pi r^2} \biggl[  \int_{C_n} dx \Bigl\{ \I\{ N_r(x)\not=d \}
             + \sum_{\ell \in \mathcal{I}_{r}(x)}^n [\I\{ M_{\ell}\not=d-1\} - \I\{ M_{\ell}\not=d \}] \Bigr\} \non \\
     &\qquad\mbox{} - \int_{B_{i,r}} dx \Bigl\{ \I\{ N_r(x)\not=d \}
             + \sum_{\ell \in \mathcal{I}_{r}(x)}^n  [\I\{ M_{\ell}=d\} - \I\{ M_{\ell}=d-1\}] \Bigr\} \biggr] \nonumber \\
     &\Eq \frac{\pi r^2}{n-\pi r^2}\, (W_{d-1} - W_{d}-H_i)-\frac{1}{n-\pi r^2}Y_d+1,\label{InsOut}
\end{align}
where 
\[
    Y_d \Def \int_{C_n} \I\{ N_r(x)=d \}\, dx.
\]
As before, if removing $V_j$ from $\mathcal{V}$ does not cause the value of $K_i$ to change, then
$R_{ij}= S_j + K_i$. 
To account for instances where~$K_i$ does change, for $i\in [n]$ and $j\in \mathcal{I}_{i,r}$, we let
\begin{align}
    E_{ij} &\Def \frac{1}{\pi r^2} \int_{B_{j,r}\setminus B_{i,r}} dx \Bigl\{ \I \{N_r(x)\not=d+1\}
                  - \I \{N_r(x)\not=d\} \notag \\
        &\qquad\mbox{}  + \sum_{\ell \in  \mathcal{I}_{r}(x)} [\I\{ M_\ell=d+1\} - 2 \I\{M_\ell = d\}
                   + \I \{M_\ell=d-1\}] \Bigr\}, \notag \\
  \begin{split}
    &\Eq \frac{1}{\pi r^2} \int_{B_{j,r}\setminus B_{i,r}} dx \Bigl\{ \I \{N_r(x)=d\} - \I \{N_r(x)=d+1\}  \\
    &\qquad\mbox{} +\sum_{\ell \in  \mathcal{I}_{r}(x)} [\I\{ M_\ell=d+1\} - 2 \I\{M_\ell = d\} + \I \{M_\ell=d-1\}] \Bigr\},
  \end{split}\label{ExpE}
\end{align}
be the change in~$K_i$ caused by removing $V_j$ from~$\mathcal{V}$. 
Using the expression for~$K_i$ in~\eqref{InsOut}, we now have
\begin{equation}\label{RExp}
    R_{ij} \Eq S_j + \frac{\pi r^2}{n-\pi r^2}\, (W_d - W_{d-1}-H_i) -\frac{1}{n-\pi r^2}\,Y_d+ E_{ij}.
\end{equation}

The following lemma is similar to \cite[Lemma 3.3]{Bart18}, and shows that the values of $S_j$, $H_i$, $Q_{ij}$ 
and~$E_{ij}$, as well as of $\sum_{j \in \mathcal{I}_{i,r}}S_j$, are uniformly bounded. 
The bounds are expressed in terms of~$\kappa_s$, where~$\kappa_s$ is the maximum number of disjoint unit balls that 
can be packed inside a ball of radius~$s$. A crude bound on~$\kappa_s$, which is sufficient for our purposes, is 
\[
   \kappa_s \Le \frac{\leb\{B_{s}(\cdot)\}}{\leb\{B_{1}(\cdot)\}} \Eq s^2,
\]
where $\leb$ denotes Lebesgue measure.

\begin{lemma}\label{GGL3}
For every $i,j$ and configuration $\mathcal{V}$,
\[
  |S_j| \leq \kappa_{3} (d+2), \quad |H_i| \leq \kappa_{3} (d+1),  \quad |Q_{ij}| \leq 2 \kappa_{3} (d+2), 
     \quad |E_{ij}| \leq 2 \kappa_{3} (d+2),
\]
and
\[
   \Bigl|\sum_{j \in \mathcal{I}_{r}(V_i)}S_j \Bigr| \Le \kappa_{5}(d+2)^2.
\]
\end{lemma}

\begin{proof}
For $s>0$, let 
\[
   \Gamma_{s,u}(x) \Def \mathcal{I}_{sr}(x) \cap \{ j \in [n]\colon\, M_j =u \}
\]
be the set of indices of all degree~$u$ vertices within distance~$sr$ of~$x$, so that 
\[
    \sum_{j \in \mathcal{I}_{sr}(x)} \I \{ M_i = u \} \Eq \card\{\Gamma_{s,u}(x)\}.
\]
To bound $\card\{\Gamma_{s,u}(x)\}$, first observe that 
\begin{equation}\label{OverEquiv}
   M_i \Eq \sum_{j \neq i} \I \{ B_{r/2}(V_i) \cap B_{r/2}(V_j) \neq \emptyset \};
\end{equation}
that is, $M_i$ is given by the number of \emph{$r/2$-grains\/} that intersect the $r/2$-grain of germ~$i$.
Now let~$\mathcal{R}_{s,u}(x)$ be a subset of $\Gamma_{s,u}(x)$ with maximal size such that its corresponding $r/2$-grains 
are pairwise disjoint.
Because the $r/2$-grains of elements in~$\mathcal{R}_{s,u}(x)$ are contained within $B_{sr+r/2}(x)$, 
we have $\card\{\mathcal{R}_{s,u}(x)\} \leq \kappa_{2s+1}$.
By the maximality of $\mathcal{R}_{s,u}(x)$, the $r/2$-grain of each germ in $\Gamma_{s,u}(x) \setminus \mathcal{R}_{s,u}(x)$ 
must intersect that of a germ in~$\mathcal{R}_{s,u}(x)$. Because the $r/2$-grain of each germ in $\mathcal{R}_{s,u}(x)$ 
intersects~$u$ other $r/2$-grains, we then have $\card\{\Gamma_{s,u}(x)\} \leq \kappa_{2s+1} (u+1)$, which implies that
\begin{equation}\label{BoGa2}
     \sum_{j \in \mathcal{I}_{sr}(x)} \I \{M_j =u \} \Le \kappa_{2s+1} (u+1).
\end{equation}
Applying \eqref{BoGa2} to expressions \eqref{ExpS}, \eqref{ExpH}, \eqref{ExpQ} and~\eqref{ExpE} yields the bounds on 
$S_j$, $H_i$, $Q_{ij}$ and~$E_{ij}$, respectively.

To bound $\sum_{j \in \mathcal{I}_{i,r}}S_j$, observe that
\begin{align*}
  \sum_{j \in \mathcal{I}_{i,r}}S_j 
         &\Le \sum_{j \in \mathcal{I}_{r}(V_i)} \sum_{\ell \in \mathcal{I}_{r}(V_j)} \I\{M_{\ell}=d\} 
    \Le \sum_{\ell \in \mathcal{I}_{2r}(V_i)} \I\{M_{\ell}=d\} \sum_{j \in \mathcal{I}_{r}(V_\ell)} 1 \\
    &\Eq \sum_{\ell \in \mathcal{I}_{2r}(V_i)} \I\{M_{\ell}=d\} (M_\ell+1) \Le  \kappa_{5} (d+1)^2.
\end{align*}
A corresponding lower bound can be obtained by applying the same arguments, but with~$d$ replaced by~$d+1$.
\end{proof}

We say that the \emph{radius\/} of a random variable indexed by~$i$ is the smallest value of~$\rho$ such that~$X_i$
is determined by the positions of $B_\rho(V_i)\cap\mathcal{V}$ relative to $V_i$. Observe that radii of $\I\{M_i=d\}$, $\pi^{(d)}_{M_i}$ and $\gamma^{(d)}_{M_i}$
are all equal to~$r$, that the radius of $S_i$ is~$2r$, that the radii of $H_i$ and $Q_{ij}$ (by \eqref{QFR}) are both~$3r$, 
and that the radius of $\sum_{j \in \mathcal{I}_{i,r}} E_{ij}$ is~$4r$. 
Because each of these random variables has a finite radius, we refer to them as \emph{local statistics\/}.
In light of Lemma~\ref{GGL3}, the next lemma demonstrates that the~$T$ arising in Theorem~\ref{BRRT}, when we use 
the size biased coupling described above, can be bounded by absolute values of centred variables of the following 
forms: sums of uniformly bounded local statistics ($T_{1},T_{2},T_{5},T_{6},T_{9},T_{11}$), $1/n$ times the products 
of sums of bounded local statistics ($T_3, T_{7},T_8, T_{10}$), and sums of terms which are only non-zero on the 
rare events $\I\{M_i > n/2\}$ ($T_4$).

\begin{lemma}\label{GGL1}
Using the notation defined above, we have 
\[
     \left| \mbE[GD| \mathcal{V} ] - \sigma^2 \right| \leq (1-b_d)\sum_{i=1}^{11} T_i,
\]
where $T_i := |T_i' - \ex T_i'|$, $1\le i\le 11$, and
\begingroup
\allowdisplaybreaks
\begin{align*}
  T_1' &\Eq q \sum_{i=1}^n \pi^{(d)}_{M_i} H_i; & 
  T_2' &\Eq q \sum_{i =1}^n \frac{\pi_{M_i}^{(d)} }{n-M_i -1} \sum_{j \in \hcI_{i,r}^c}  Q_{ij};  \\
  T_3' &\Eq \frac{2 q}{n} \sum_{i =1}^n \frac{n\pi_{M_i}^{(d)} \I\{ M_i \leq n/2\} }{2(n-M_i -1)} \sum_{j=1}^n S_j; &
  T_4' &\Eq  q \sum_{i =1}^n \frac{\pi_{M_i}^{(d)} \I\{ M_i > n/2\} }{n-M_i -1} \sum_{j=1}^n S_j; \\ 
  T_5' &\Eq q \sum_{i=1}^n \frac{\pi^{(d)}_{M_i}}{n-M_i-1} \sum_{j \in \mathcal{I}_{r}(V_i)} S_j ; &
  T_6' &\Eq (1-q) \sum_{i=1}^n  \frac{ \gamma^{(d)}_{M_i}}{M_i}\sum_{j \in \mathcal{I}_{i,r}} S_j ; \\
  T_7' &\Eq \frac{ (1-q) \pi r^2}{n - \pi r^2} \, W_d \sum_{i=1}^n \gamma^{(d)}_{M_i} ; &
  T_8' &\Eq \frac{ (1-q) \pi r^2}{n - \pi r^2} \, W_{d-1} \sum_{i=1}^n \gamma^{(d)}_{M_i} ; \\
  T_9' &\Eq  \frac{ (1-q) \pi r^2}{n - \pi r^2} \sum_{i=1}^n \gamma^{(d)}_{M_i} H_i; &
  T_{10}' &\Eq \frac{ (1-q)}{n - \pi r^2}  Y_d \sum_{i=1}^n \gamma^{(d)}_{M_i}; \\
  T_{11}' &\Eq  (1-q) \sum_{i=1}^n \frac{\gamma^{(d)}_{M_i}}{M_i} \sum_{j \in \mathcal{I}_{i,r}}E_{ij}, 
\end{align*}
\endgroup
adopting the convention that summands with zero numerator are zero, for example if $M_i=0$ in~$T_{11}$.
\end{lemma}

\begin{proof} 
Applying \eqref{AExp} and \eqref{RExp} in \eqref{eqn:46} and \eqref{eqn:410}, respectively, we obtain
\begin{align}
   \mbE(G&D \giv \mathcal{V}) \Eq \mu \mbE(W^s_d - W_d \giv \mathcal{V}) \nonumber \\
  &\Eq n(1-b_d) \sum_{i=1}^n \frac{1}{n}\left[\frac{q \pi^{(d)}_{M_i}}{n-M_i-1}\sum_{j \in \hcI_{i,r}^c} A_{ij} 
          + \frac{(1-q) \gamma^{(d)}_{M_i}}{M_i} \sum_{j \in \mathcal{I}_{i,r}} R_{ij}\right]\nonumber \\
  &\Eq (1-b_d)q\, \sum_{i=1}^n \frac{ \pi^{(d)}_{M_i}}{n-M_i-1}\sum_{j \in \hcI_{i,r}^c} (H_i +S_j +Q_{ij})
         \label{eqn:46}\\
   \begin{split}\label{eqn:410}
	&\qquad\mbox{} + (1-b_d)(1-q)\sum_{i=1}^n \frac{ \gamma^{(d)}_{M_i}}{M_i} 
             \sum_{j \in \mathcal{I}_{i,r}} \bbclr{S_j + \frac{\pi r^2}{n-\pi r^2} (W_d - W_{d-1}-H_i)  \\
	&\hspace{10cm}\mbox{} -\frac{1}{n-\pi r^2}Y_d + E_{ij} }                                                                                    
   \end{split} \\
  &\Eq (1-b_d)q \biggl[\sum_{i=1}^n \pi^{(d)}_{M_i} H_i  + \sum_{i =1}^n \frac{\pi_{M_i}^{(d)} }{n-M_i -1} 
       \sum_{j \in \hcI_{i,r}^c} S_j+  \sum_{i =1}^n \frac{\pi_{M_i}^{(d)} }{n-M_i -1} 
                                       \sum_{j \in  \hcI_{i,r}^c}  Q_{ij}\biggr] \label{deqz2}\\
  &\qquad\mbox{} + (1-b_d)(1-q) \biggl[\sum_{i=1}^n \frac{ \gamma^{(d)}_{M_i}}{M_i} \sum_{j \in \mathcal{I}_{i,r}} S_j 
              + \frac{\pi r^2(W_d - W_{d-1})}{n-\pi r^2}\sum_{i=1}^n \gamma^{(d)}_{M_i}\nonumber  \\
  &\qquad\qquad\qquad\mbox{} - \frac{\pi r^2}{n-\pi r^2}\sum_{i=1}^n \gamma^{(d)}_{M_i} H_i 
       - \frac{Y_d}{n-\pi r^2} \sum_{i=1}^n \gamma^{(d)}_{M_i} + 
           \sum_{i=1}^n \frac{ \gamma^{(d)}_{M_i}}{M_i} \sum_{j \in \mathcal{I}_{i,r}} E_{ij} \biggr]. \nonumber 
\end{align}
We rearrange the second term in \eqref{deqz2} to obtain
\begin{align*}
    \sum_{i =1}^n \frac{\pi_{M_i}^{(d)} }{n-M_i -1} \sum_{j \in \hcI_{i,r}^c} S_j 
    &  \Eq  \sum_{i =1}^n \frac{\pi_{M_i}^{(d)} }{n-M_i -1} \left( \sum_{j=1}^n S_j -  
                                              \sum_{j \in \mathcal{I}_{r}(V_i)} S_j \right) \\[1ex]
    &\Eq \sum_{i =1}^n \frac{\pi_{M_i}^{(d)} \I\{ M_i \leq n/2\} }{n-M_i -1} \sum_{j=1}^n S_j  
                 + \sum_{i =1}^n \frac{\pi_{M_i}^{(d)} \I\{ M_i > n/2\} }{n-M_i -1} \sum_{j=1}^n S_j
  \\  &\qquad\mbox{} 
    - \sum_{i =1}^n \frac{\pi_{M_i}^{(d)} }{n-M_i -1} \sum_{j \in \mathcal{I}_{r}(V_i)} S_j.
\end{align*}
After observing that $\sigma^2=\mbE[GD]$, the result follows by applying the triangle inequality. 
\end{proof}

The next lemma is used to show that the quantities~$T_i'$ are weakly concentrated about their means. 
Note that, in view of Lemma~\ref{GGL3}, we only apply this result when $\beta=0$; however, proving the result in the 
more general form stated below requires little additional effort.

\begin{lemma}\label{GGL2}
Let $\mathrm{U}$ be a real-valued function on finite subsets of~$C_n$ with a distinguished point~$v$, whose value is determined
by the positions of points relative to~$v$. Suppose also that there 
exist constants $c,\beta \geq 0$ such that, for any set $A \subset C_n$ with distinguished point, we have
\begin{equation}\label{Vbound}
             |\mathrm{U}(A,v)| \leq c\, \card\{A\}^\beta.
\end{equation}
Fixing $s>0$, define
\[
   X_i \Def \mathrm{U}(\mathcal{V} \cap B_{i,s}, V_i )\qquad\mbox{and} \qquad W \Def \sum_{i=1}^n X_i.
\]
Then, for any $q\in \IN$, 
\[
    n^{-1/2} \norm{W-\mbE W}_q \Le 2\sqrt{6} c[ C_A \max\{9 \pi s^2, q(1+\beta) \}+1]^{\beta + 3/2},
\]
where $C_A:= \pi e^{e-2}/\log(e-1)$.
\end{lemma} 

\begin{proof}
To obtain moment bounds for~$W$, we use Theorem~\ref{Yuting}, together with a suitable Stein coupling that makes use
of the dependence structure. 
For each $j = 1, \dots, n$ we generate a new configuration 
$\widetilde{ \mathcal{V}}^{(j)}=\{\widetilde{V}^{(j)}_1 , \dots, \widetilde{V}^{(j)}_n \}$ from~$\mathcal{V}$. 
To generate $\widetilde{\mathcal{V}}^{(j)}$, for each $i \in \cI_{j,s}$ we let $\widetilde{V}_i^{(j)}$ be uniformly 
distributed on~${C}_n$, independently of everything else; for each 
$i \in \hcI^c_{j,s}=\{1,\dots, n\} \setminus (\mathcal{I}_{j,s} \cup \{j\})$, we let $\widetilde{V}_i^{(j)}=V_i$ with 
probability $1-\pi s^2/n$, and otherwise let $\widetilde{V}_i^{(j)}$ be uniformly distributed on $B_{j,s}$, where both 
randomizations occur independently of everything else; finally, we let $\widetilde{V}^{(j)}_j=V_j$.
Let~$J$ be uniform on the set $\{1, \dots, n \}$, independently of the random objects above, and define 
$\widetilde W:= \widetilde{W}^{(J)}$ and $\tG :=-n(X_J - \mbE X_J)$. 
To show that $(W, \widetilde W, \tG)$ is an exact Stein coupling, first observe that for each $j \in [n]$,  $\{ V_i : i \in \mathcal{I}_{j,s} \}$ and~$\widetilde{\mathcal{V}}^{(j)}$ are conditionally independent given $V_j$. Therefore,~$X_j$ and~$\wt W^{(j)}$ are also conditionally independent given~$V_j$.   Since~$X_j$ depends 
on the positions of the points of$~\cV$ only \emph{relative} to~$V_j$, it follows that $\law(X_j| V_j)=\law(X_j)$, and thus combining the arguments above, that $X_j$ is independent of $\widetilde W^{(j)}$. Consequently,
\begin{align*}
\mbE[\tG( &f(\tW)-f(W))] = \mbE[-n (X_J - \mbE X_J) (f(\tW^{(J)}) -f(W))] \\
&= \sum_{j =1}^n -\mbE[(X_j - \mbE X_j) (f(\tW^{(j)}) -f(W))] = \mbE[ (W- \mbE W) f(W) ];
\end{align*}
that is, $(W, \widetilde W, \tG)$ is an exact Stein coupling. 
Moreover,
because $\mathcal{L}(\widetilde W) = \mathcal{L}(W)$, the central moments of $\widetilde W$ and~$W$ are equal, and hence 
Theorem \ref{Yuting} applies.


To apply Theorem \ref{Yuting} to bound the $q$-th central moment, we need to bound 
\begin{equation}\label{GGG}
   \norm{\tG}^q_q \Eq n^{q-1} \sum_{i=1}^n \mbE|X_i - \mbE X_i |^q,
\end{equation}
and, for $\tD := \tW - W$, 
\begin{align}
    \norm{\tD}^q_q &\Eq \frac{1}{n} \sum_{j=1}^n \mbE | \widetilde{W}^{(j)} -W |^q \nonumber \\
    &\Eq \frac{1}{n} \sum_{j=1}^n \mbE \Bigl| \sum_{i=1}^n (\widetilde{X}^{(j)}_i - X_i ) \Bigr|^q 
          \Eq \mbE \Bigl| \sum_{i=1}^n (\widetilde{X}_i^{(1)}-X_i) \Bigr|^q. \label{GGD}
\end{align}

To bound \eqref{GGG}, note that, in view of \eqref{Vbound} and Lemma~\ref{BinMom},
\begin{equation}\label{GGG1}
   \norm{X_i}_q \Le c \norm{N_s(V_i)^\beta}_q \Eq c \norm{N_s(V_i)}_{\beta q}^\beta 
                \Le c [C_A \max\{\pi s^2, q \beta\}+1]^{\beta}.
\end{equation}
Using \eqref{GGG1} and Minkowski's inequality, this gives
\[
   \mbE| X_i - \mbE X_i |^q \Eq \norm{ X_i - \mbE X_i}^q_q \leq (\norm{X_i}_q + \norm{X_i}_1 )^q 
                            \Le (2c)^q [C_A \max\{\pi s^2, q \beta\}+1]^{\beta q}.
\]
Thus, from \eqref{GGG}, we can bound 
\begin{equation}\label{GBD}
     \norm{\tG}_q \Le 2 n c(C_A \max\{\pi s^2, q \beta\}+1)^{\beta}.
\end{equation}

To bound \eqref{GGD}, for $\ell >0$, let 
$$
   \widetilde{\mathcal{I}}^{(1)}_{j,\ell} 
          \Def \left\{ i \neq j: \widetilde V^{(1)}_i \in  B_\ell(\widetilde V_j^{(1)}) \right\}
$$ 
contain the indices of the germs located in $B_\ell (\widetilde V^{(1)}_j)$ in configuration $\widetilde{\mathcal{V}}^{(1)}$. 
Observe that $|X_i-\widetilde X^{(1)}_i| = 0$ if  germ~$i$ is located $(i)$ at least~$2s$ units from~$V_1$, $(ii)$  
at least~$s$ units from  the previous location of a germ that has been moved inside~$B_s(V_1)$, and $(iii)$ at least~$s$ 
units from any germ which has been moved outside~$B_s(V_1)$.
 
We define $\cM$ as follows, so that $\{1, \dots, n\} \setminus \mathcal{M}$ contains the indices of such germs:
\begin{align}
    \mathcal{M} &\Eq \left( \{1\} \cup \mathcal{I}_{1,2s} \right)  
           \bigcup \left( \cup_{i \in \widetilde{\mathcal{I}}^{(1)}_{1,s}} \mathcal{I}_s(V_i) \right)  
              \bigcup \left( \cup_{i \in \mathcal{I}_{1,s}} \mathcal{I}_{s}(\widetilde V_i^{(1)}) \right)\nonumber \\
   &\Eq \left( \{1\} \cup \mathcal{I}_{1,2s} \right)  
               \bigcup \left( \cup_{i \in \widetilde{\mathcal{I}}^{(1)}_{1,s}} (\mathcal{I}_{i,s} \cup \{i\}) \right)  
               \bigcup \left( \cup_{i \in \mathcal{I}_{1,s}} \widetilde{\mathcal{I}}^{(1)}_{i,s} \right) \label{Mexp1} \\
   &\Eq \left( \{1\} \cup \widetilde{\mathcal{I}}^{(1)}_{1,2s} \right) 
           \bigcup \left( \cup_{i \in {\mathcal{I}}_{1,s}} (\widetilde{\mathcal{I}}^{(1)}_{i,s} \cup \{i\}) \right) 
               \bigcup \left( \cup_{i \in \widetilde{\mathcal{I}}^{(1)}_{1,s}} {\mathcal{I}}_{i,s} \right), \label{Mexp2} 
\end{align}
where the last equality is  by considering the the groups $(i), (ii), (iii)$ relative to the new configuration $\wt\cV^{(1)}$.
Even though these sets are equivalent, we think of \eqref{Mexp1} as~$\mathcal{M}$ and \eqref{Mexp2} 
as~$\widetilde{\mathcal{M}}^{(1)}$. We now have
\begin{align}
 \left| \sum_{i=1}^n (\widetilde X_i^{(1)}-X_i) \right| 
        &\Eq \left| \sum_{i\in \mathcal{M}} (X_i - \widetilde X_i^{(1)}) \right| 
          \Le \sum_{i\in \mathcal{M}} |X_i| +  \sum_{i\in \widetilde{\mathcal{M}}^{(1)}} |\widetilde X^{(1)}_i|  \non \\
    &\Le \sum_{i \in \mathcal{I}_{1,2s} \cup \{1\}} |X_j| + 
             \sum_{i \in \widetilde{\mathcal{I}}^{(1)}_{1,2s}\cup \{1\}} |\widetilde X_j^{(1)}| \label{DE1} \\
   &\quad\mbox{} +  \sum_{i \in \widetilde{\mathcal{I}}^{(1)}_{1,s}} \sum_{j \in \mathcal{I}_{i,s}\cup \{i\}} |X_j|
          + \sum_{i \in \mathcal{I}_{1,s}} \sum_{j \in \widetilde{\mathcal{I}}^{(1)}_{i,s} \cup \{i\}} |\widetilde X_j^{(1)}| 
               \label{DE2}\\
   &\quad\mbox{}  + \sum_{i \in \mathcal{I}_{1,s}} 
      \sum_{j \in \widetilde{\mathcal{I}}^{(1)}_{i,s}\setminus (\mathcal{I}_{1,2s} \cup \{1\})} |X_j| 
          + \sum_{i \in \widetilde{\mathcal{I}}^{(1)}_{1,s} } 
        \sum_{j \in \mathcal{I}_{i,s}\setminus (\widetilde{\mathcal{I}}^{(1)}_{1,2s} \cup \{1\})} |\widetilde X^{(1)}_j| 
               \label{DE3},
 \end{align}
where we remove some double counting in \eqref{DE3}.
By considering the process of generating $\wt\cV^{(1)}$ from $\cV$ in reverse, observe that the left and right hand 
terms of \eqref{DE1}--\eqref{DE3} have the same distributions, and thus we only need to bound the $q$th-moment of one 
term of each.

To bound the first term of \eqref{DE1}, observe that, for each $i \in \mathcal{I}_{1,2s}$,
\[
    |X_i| \Le c (N_{s}(V_i))^\beta \Le  c( N_{3s}(V_1) )^\beta,
\] 
and that there are $N_{2s}(V_1)$ such indices; hence
\begin{equation*}
      \sum_{i \in \mathcal{I}_{1,2s} \cup \{1\}} |X_j| \Le c  N_{2s}(V_1)  (N_{3s}(V_1) )^\beta 
                     \Le c (N_{3s}(V_1) )^{\beta+1} ,
\end{equation*}
implying that
\begin{align}
   \bbbnorm{ \sum_{i \in \mathcal{I}_{1,2s} \cup \{1\}} |X_j|}_q &\Le c\norm{ (N_{3s}(V_1))^{\beta+1}}_{q} 
    \Eq c \norm{N_{3s}(V_1) }_{(\beta+1)q}^{\beta+1}. \nonumber 
\end{align}
Because $\mathcal{L}(N_{3s}(V_1)) = Bi(n-1, 9 \pi s^2/n)+1$, we may then apply Lemma~\ref{BinMom} to the right hand side of this last display to obtain
\begin{align}
  \bbbnorm{ \sum_{i \in \mathcal{I}_{1,2s} \cup \{1\}} |X_j|}_q &\Le c (C_A \max \{9 \pi s^2, (\beta+1)q \}+1)^{\beta+1}.
                 \label{DE1B}
\end{align}
As previously mentioned, this also bounds the second term of~\eq{DE1}.

To bound the second term of \eqref{DE2}, for $x \in C_n$, let $\widetilde N_s^{(1)}(x)$ be the number of points 
in configuration $\widetilde{\mathcal{V}}^{(1)}$ that fall in $B_{s}(x)$. Observe that, by the same arguments as for~\eq{DE1}, 
\[
     \sum_{i \in \mathcal{I}_{1,s}} \sum_{j \in \widetilde{\mathcal{I}}^{(1)}_{i,s} \cup \{i\}} |\widetilde X_j^{(1)}| 
                \Le \sum_{i \in \mathcal{I}_{1,s}} c(\widetilde{N}^{(1)}_{2s}(\widetilde V^{(1)}_i))^{\beta+1}.
\]
By construction, $(\widetilde{N}^{(1)}_{2s}(\widetilde V^{(1)}_i))_{i\in \mathcal{I}_{1,s}}$ has the same distribution
as the counts of the first (say) $\abs{\cI_{1,s}}$ $2s$-neighbourhoods in an independent uniform $n$-configuration, and so  we can  apply Lemma \ref{Branch1}. Applying Lemmas \ref{Branch1} and \ref{BinMom} thus gives
\begin{align}
   \bbbnorm{\sum_{i \in \mathcal{I}_{1,s}} c(\widetilde{N}^{(1)}_{2s}(\widetilde V^{(1)}_i))^{\beta+1}}_q 
      &\Le c\norm{N_{s}(V_1)-1}_q\norm{(\widetilde{N}^{(1)}_{2s}(\widetilde V^{(1)}_i))^{\beta+1}}_q  \nonumber \\
      &\Le c (C_A \max \{4 \pi s^2, (\beta+1)q \}+1)^{\beta+2} .  \label{DE2B}
\end{align}

To bound the first term of  \eqref{DE3}, first note that, for each $i \in \mathcal{I}_{1,s}$ and 
$j \in \mathcal{\widetilde I}^{(1)}_{i,s} \setminus \mathcal{I}_{1,2s}$, we have 
$|X_j| \leq c (N_{2s}(\widetilde V^{(1)}_i))^\beta$ (recall that $N_{2s}(\widetilde V^{(1)}_i)$ gives the number 
of germs that fall in $B_{2s}(\widetilde V^{(1)}_i)$ \emph{in configuration\/} $\mathcal{V}$, \emph{not\/} 
$\mathcal{\widetilde V}^{(1)}$). However, because $N_{2s}(\widetilde V^{(1)}_i)$ and $\mathcal{I}_{1,s}$ are dependent 
(to see why, consider the distribution of $N_{2s}(\widetilde V^{(1)}_i)$ given $\card\{\mathcal{I}_{1,s}\}=n$), 
we are unable to apply Lemma~\ref{Branch1} directly.
To construct a bound on $\sum_{j\in\mathcal{\widetilde I}^{(1)}_{i,s} \setminus (\mathcal{I}_{1,2s}\cup \{1\})} |X_j|$ 
that is independent of $\mathcal{I}_{1,s}$, first note that, for each $i \in \mathcal{I}_{1,s}$ and 
$j \in \mathcal{\widetilde I}^{(1)}_{i,s} \setminus (\mathcal{I}_{1,2s}\cup\{1\})$, we have 
\[
    |X_j| \Le c (N_s(V_j))^\beta \Eq c\,\card\{\mathcal{I}_{s}(V_j) \setminus \mathcal{I}_{s}(V_1)\}^\beta,
\]
since $\cI_s(V_j)\cap\cI_s(V_1) = \emptyset$ when $J \notin \cI_{1,2s}\cup\{1\}$.  Then, because 
$j \in \mathcal{\widetilde I}^{(1)}_{i,s}$, $V_j \in B_s(\tV_i^{(1)})$, and so $\cI_s(V_j) \subset \cI_{2s}(\tV_i^{(1)})$;
hence
\[
   |X_j| \Le c\,\card\{\mathcal{I}_{2s}(\widetilde V^{(1)}_i) \setminus \mathcal{I}_{s}(V_1)\}^\beta.
\]

We now construct a new configuration, $\widehat {\mathcal{V}}^{(1)}=\{\widehat V^{(1)}_1, \dots, \widehat V^{(1)}_n \}$,
from $\mathcal{V}$, such that, for $i \in \mathcal{I}^c_{1,s}$, $\widehat V^{(1)}_i = V_j$ and, for 
$i \in \mathcal{I}_{1,s} \cup \{1\}$, $\widehat V^{(1)}_i$ is distributed uniformly on $C_n \setminus B_{i,s}$, 
independently of everything else. Extending our notation naturally, we now have 
\[
   |X_j| \Le c\,\card\{ \mathcal{I}_{2s}(\widetilde V^{(1)}_i) \setminus \mathcal{I}_{s}(V_1)\}^\beta 
            \Le c(\widehat N_{2s}(\widetilde V^{(1)}_i))^\beta.
\]
By construction, $\widehat N_{2s}(\widetilde V^{(1)}_i)$ is independent of $\mathcal{I}_{1,s}$ and is stochastically 
dominated by the distribution $\Bi(n, 4\pi s^2/(n-4\pi s^2))$.
We may now apply Lemmas \ref{Branch1} and~\ref{BinMom} to obtain
\begin{align}
   \bbbnorm{\sum_{i \in \mathcal{I}_{1,s}} \sum_{j \in \mathcal{I}^{(1)}_{i,s}\setminus \mathcal{I}_{1,2s} } |X_j|}_q 
    &\Le c\norm{N_s(V_1)-1}_q\, \norm{\widehat{N}_{2s}(V^{(1)}_\ell)}^{\beta+1}_{q(\beta+1)} \nonumber \\
    &\Le c (C_A \max \{4 \pi s^2, (\beta+1)q \}+1)^{\beta+2}.  \label{DE3B}
\end{align}

Combining \eqref{GGD} with \eqref{DE1}--\eqref{DE3B} we obtain
\begin{align}
    \norm{\tD}_q &\Le 6c (C_A \max \{9 \pi s^2, (\beta+1)q \}+1)^{\beta+2}. \label{DBF}
\end{align}
Thus, using \eqref{GBD} and \eqref{DBF}, we can apply Theorem \ref{Yuting} to obtain the required result.
\end{proof}

\begin{corollary}\label{GGmoments}
 For each $1 \le l\le 11$ and for $ q \in \{2, \lceil\log\s\rceil\}$, we have $\s^{-1}\norm{T_l}_q \le bq^{3/2}$,
for a suitable fixed choice of~$b$.
\end{corollary}

\begin{proof}. 
For $i =1,2,5,6,9,11$, Lemma~\ref{GGL3} implies that the conditions of Lemma~\ref{GGL2} hold with $s=4r$, 
$c=\kappa_{5}(d+2)^2$ and $\beta=0$, so that there are constants $b_i$, $i =1,2,5,6,9,11$, not depending on $\sigma$, 
such that, for any $q\in\IN$,
\ben{\label{eq:ts1germ}
     \sigma^{-1}\norm{ T_i}_q \Le b_i q^{3/2}.
}
For $i=3,7,8,10$, we use~\Ref{CSXY}, taking $X=\sum_{i=1}^n \frac{n\pi_{M_i}^{(d)} \I\{ M_i \leq n/2\} }{2(n-M_i -1)}, 
\; W_d,\; W_{d-1}, \; Y_d$ and $Y= \sum_{j=1}^n S_j, \; \sum_{i=1}^n \gamma^{(d)}_{M_i}, \; \sum_{i=1}^n \gamma^{(d)}_{M_i}, 
\; \sum_{i=1}^n \gamma^{(d)}_{M_i}$, respectively;
with the exception of~$Y_d$, we can apply Lemma~\ref{GGL2} with $s=2r$, $c=\kappa_{3}(d+2)$ and $\beta =0$ to prove 
that there is a constant $\hat c>0$ such that, for all $k\in\IN$, 
\ben{\label{eq:bddds}
  \sigma^{-1}\norm{X-\mbE X}_{k} \Le  \hat c k^{3/2}, \,\,\, \mbox{ and } \,\,\,
   \sigma^{-1}\norm{Y-\mbE Y}_{k} \Le \hat c k^{3/2}.
} 
For~$Y_d$, rather than defining a Stein coupling and using Theorem~\ref{BRRT}, we quote \cite[Theorem 3.2]{Bart18}, 
which implies that, for some $C>0$ and all $t>0$,
\begin{equation}\label{BGGV}
    \mbP(|Y_d - \mbE Y_d | > t) \Le \exp \left( - \frac{t^2}{2C \mbE Y_d}\right). 
\end{equation}
Since $\mbE Y_d \leq n$, a standard calculation implies that~\eq{eq:bddds} also holds for $Y=Y_d$.
Now applying~\eqref{CSXY}, and using the fact that, for $i=3,7,8,10$, $n^{-1} \mbE X \leq 1$ and 
$n^{-1} \mbE Y \leq \kappa_{3/2}(d+1)$, we deduce that there are
constants $b_i, i=3,7,8,10$ such that
\ben{\label{eq:ts2germa}
    \sigma^{-1}\norm{T_i}_2 \leq b_i,
}
and, for $q=\ceil{\log(\sigma)}$, and for all~$\s$ large enough,
\ben{\label{eq:ts2germb}
   \sigma^{-1}\norm{T_i}_q \Le b_i \max\left\{\sigma^{-1} (\log\sigma)^3, (\log\sigma)^{3/2}\right\}
         \Le b_i \log(\sigma)^{3/2}.
}
To bound $T_4$, we use \Ref{Chernoff}, which implies that, for some constant $C = C_{r,d}$, 
\ben{
     \mbP \left( \sum_{i=1}^n \I\{M_i > n/2\} \neq 0 \right) \Le Cn\exp(-n/6).
}
In addition, applying Lemma~\ref{GGL3} to bound $S_{(\cdot)}$, we have  
\ben{
\left| \sum_{i =1}^n \frac{\pi_{M_i}^{(d)} \I\{ M_i > n/2\} }{n-M_i -1} \sum_{j=1}^n S_j  \right| \Le n^2 \kappa_3(d+1)
  \quad{\rm a.s.},
}
since the interpretation of any summand with $M_i = n-1$ is zero, as remarked in Lemma~\ref{GGL2}.
Thus $\sigma^{-1} \norm{T_4}_q\leq b_4$ for all $q\in\IN$, and the corollary is proved.
\end{proof}

 \nin
Hence, in particular, when applying Theorem~\ref{BRRT}, we can take $c_2 = 11 b\ceil{\log\s}^{3/2}$ in~\eqref{eq:ezbd2}, 
and $\s^{-1}\norm{T}_2 \le 11\times 2^{3/2}b$ in~\eqref{eq:upsrti2bd}.

Since $R=0$ a.s.,
all that remains to be proved, in order to apply Theorem~\ref{BRRT}, is a bound of order~$\bigo(1)$ for the smoothness 
term~$\Upsilon$, to be used in \Ref{eq:upsrti2bd}.
Using the notation $I$, $J$, $X$ and~$V_j^s$ from around~\Ref{Xdefi}, we define
\[                                                           
   \mathcal{F}_2 \Def \sigma(I, X, J,  V_I,  V_J, V^s_J, \mathcal{V} \cap B_{2r}(V_I),  \mathcal{V} \cap B_{r}(V_J),  
         \mathcal{V} \cap B_{r}(V^s_J));
\]
observe that $D$ (and $G$) are $\mathcal{F}_2$-measurable. 
We adopt the convention that $J=I$ when $X=0$.

\begin{lemma}\label{GGL5}
    For $\Upsilon$ defined in Theorem \ref{BRRT}, we have $\Upsilon=O(1)$.
\end{lemma}

\begin{proof}
From Lemma \ref{GGL3}, we see that there is a constant~$K$ such that $|D| \leq K$, uniformly in~$n$ (see also the 
proof of \cite[Theorem 3.3]{Bart18}). 
Letting $N^{\mathcal{F}_2}:= \card\{\mathcal{V} \cap ( B_{2r}(V_I) \cup B_r(V_J) \cup B_r(V_J^s) )\}$, we have
\begin{align}
\Upsilon  &= \mbE\left[ | G D (D-1) | S_2( \mathcal{L}(W_d \giv \mathcal{F}_2) \right] \nonumber \\
  & \Le \mu K(K+1) \left\{ \mbE \left[ S_2( \mathcal{L}(W_d \giv \mathcal{F}_2) \I\{N^{\mathcal{F}_2} \leq \sqrt{n}\} \right]
    + \mbE \left[ S_2( \mathcal{L}(W_d \giv \mathcal{F}_2) \I\{ N^{\mathcal{F}_2} > \sqrt{n}\} \right]  \right\}  \non \\
  & \Le \mu K(K+1) \mbE \left[ S_2( \mathcal{L}(W_d \giv \mathcal{F}_2) \I\{ N^{\mathcal{F}_2} \leq \sqrt{n}\} \right] 
           \label{GL31} \\
  &\qquad \mbox{} +  \mu K(K+1) \mbP (N^{\mathcal{F}_2} > \sqrt{n} ). \label{GL32}
\end{align}

To bound \eqref{GL32}, observe that $\mathcal{L}(|\mathcal{V} \cap B_{2r}(V_I)|)=Bi(n-1, 4 \pi r^2)+1$.  If $X=1$, 
then $\mathcal{V} \cap B_r(V^s_J) \subseteq \mathcal{V} \cap B_{2r}(V_I)$ and 
$\card\{\mathcal{V} \cap B_r(V_J) \setminus B_{2r}(V_I)\}$ is stochastically smaller than a random variable with 
the distribution $\Bi(n-1, \pi r^2/(n-\pi r^2))+1$; and if $X=-1$, then 
$\mathcal{V} \cap B_r(V_J) \subseteq \mathcal{V} \cap B_{2r}(V_I)$ and 
$\card\{\mathcal{V} \cap B_r(V^s_J) \setminus B_{2r}(V_I)\}$ is stochastically smaller than a random variable 
with the distribution $\Bi(n-1, \pi r^2/(n-\pi r^2))+1$.
Letting $Y$ denote a random variable with distribution $\Bi(n-1, \pi r^2/(n-\pi r^2))$, and applying~\Ref{Chernoff},
we obtain 
\begin{align}
    \mbP (N^{\mathcal{F}_2} > \sqrt{n} ) &\Le \mbP( N_{2r}(V_I) \geq \sqrt{n}/2) +  \mbP( Y^{(J)} \geq \sqrt{n}/2) 
             \Le C_re^{-\sqrt n/6}, \label{McDGG}
\end{align}
for some $C_r>0$, uniformly in $n$. This implies that \eqref{GL32} is of order~$\bigo(1)$.

To bound \eqref{GL31}, given $\mathcal{F}_2$, we define a new configuration 
$\mathcal{V}^{\mathcal{F}_2} = \{ V^{\mathcal{F}_2}_1, \dots ,  V^{\mathcal{F}_2}_n \}$ by letting 
$V^{\mathcal{F}_2}_i = V_i$ if $i \in \mathcal{I}_{2r}(V_I) \cup \mathcal{I}_{r}(V_J) \cup \mathcal{I}_{r}(V^s_J)$, 
and otherwise letting $V^{\mathcal{F}_2}_i$ be uniformly distributed on 
$C_n \setminus (B_{2r}(V_I) \cup B_{r}(V_J) \cup B_r(V^s_J))$, independently of everything else. Observe that, 
if we let $W_d^{\mathcal{F}_2}$ be the number of germs in $\mathcal{V}^{\mathcal{F}_2}$ whose grain does not contain 
precisely~$d$ germs, then $\mathcal{L}( W_d^{\mathcal{F}_2}) = \mathcal{L}(W_d \giv {\mathcal{F}_2})$. To establish 
the bound on~$\Upsilon$, we prove that, for any fixed event in~$\mathcal{F}_2$ with $N^{\mathcal{F}_2} \leq \sqrt{n}$, 
we have
\begin{equation}\label{PG2}
    S_2( \mathcal{L}(\tilde W_d \giv \mathcal{F}_2))  \Eq \bigo(n^{-1}).
\end{equation}
For ease of notation, during the remainder of the proof, we tacitly assume that every random quantity has
distribution conditional on~$\mathcal{F}_2$.

We first establish \eqref{PG2} under the assumption that $d\geq 1$. Divide the space~$C_n$ into disjoint rectangles~$R_i$ 
with height $7r/3$ and width~$13r/3$; if we ignore any left over space, then there are 
$ \left\lfloor \frac{\sqrt{n}}{ 7r/3}\right\rfloor \left\lfloor\frac{\sqrt{n}}{13r/3} \right\rfloor$ such rectangles. 
Let $\mathcal{R} := \{ i\colon\, R_i \cap (B_{2r}(V_I) \cup B_{r}(V_J) \cup B_r(V^s_J)) = \emptyset \}$ be the set of 
rectangles that do not intersect $B_{2r}(V_I) \cup B_{r}(V_J) \cup B_r(V^s_J)$. Letting $n_r:= \card\{\mathcal{R}\}$, 
we see that
\begin{equation}\label{RecNum}
    \left\lfloor \frac{\sqrt{n}}{ 7r/3}\right\rfloor \left\lfloor\frac{\sqrt{n}}{13r/3} \right\rfloor -10 \leq n_r 
           \Le \left\lfloor \frac{\sqrt{n}}{ 7r/3}\right\rfloor \left\lfloor\frac{\sqrt{n}}{13r/3} \right\rfloor,
\end{equation}
since, for instance, a ball of radius~$2r$ can intersect at most~$6$ of the rectangles;
this means that~$n_r$ is of strict order~$n$. 
Let $(x_i,y_i) \in C_n$ denote the coordinates of the bottom lefthand corner of rectangle $i$, and, for $j=1,\dots,4$, 
let $c_{i,j}=(x_i+r/2 + 2rj/3,\, y_i+7r/6)$. 
We say that the \emph{good event} occurs in $R_i$, denoted~$GE_i$, if \emph{(i)} $N_{r/6}(c_{i,1}) =2$, \emph{(ii)} 
$N_{r/6}(c_{i,2})=d$, \emph{(iii)}~$N_{r/6}(c_{i,3}) + N_{r/6}(c_{i,4}) = 1$, and \emph{(iv)} 
$\cV\cap\bclr{R_i \setminus \cup_{j=1}^4 B_{r/6}(c_{i,j})}=\emptyset$.
Observe that, if the good event occurs in~$R_i$, then the total number of germs in~$R_i$ whose grain contains
exactly~$d$ germs is 1 if $N_{r/6}(c_{i,3})=1$, and is 0 if $N_{r/6}(c_{i,4})=1$, regardless of the configuration 
of germs outside~$R_i$. 

Let $\I\{GE_i\}$ denote the indicator of~$GE_i$, and 
 let $\mathcal{R}_{GE} := \{ i \in \mathcal{R} \colon\, \I\{GE_i\}=1\}$,  $N_{GE} := \card\{\mathcal{R}_{GE}\}$;
define $U_{I,J}=\leb\{B_{2r}(V_I) \cup B_{r}(V_J) \cup B_r(V^s_J)\} \le 5\pi r^2$, noting that it is $\cF_2$-measurable.
For $i \in \mathcal{R}$, through elementary calculations, we obtain
\begin{align*}
   \xi \Def \mbE \I\{GE_{i}\} &\Eq {n - N^{\mathcal{F}_2} \choose d+3} \left( \frac{|R_i|}{n-U_{I,J}} \right)^{d+3} \\
    &\qquad \times \left( \frac{n-U_{I,J}-|R_i|}{n-U_{I,J}} \right)^{n-N^{\mathcal{F}_2}-(d+3)} \frac{(d+3)!}{d!} 
          \left( \frac{\pi (r/6)^2}{|R_i|} \right)^{d+3}.
\end{align*}
Under the assumption that $N^{\mathcal{F}_2} \leq \sqrt{n}$,  we conclude that~$\xi$ converges to a positive constant,
for~$d$ fixed, as $n\to\infty$. When combined with~\eqref{RecNum}, this gives $\mbE(N_{GE})=n_r \xi \asymp n$.
We can then (suppressing the conditioning on~$\cF_2$ in the expectations) write 
\begin{align}
  S_2(W_d \giv \cF_2) &\Eq \sup_{h: \norm{h}\leq 1} \mbE[ \Delta^2 h(W_d)] \nonumber \\
     &\Eq \sup_{h: \norm{h}\leq 1} \mbE[ \Delta^2 h(W_d)\,\I\{| N_{GE} - \mbE(N_{GE}) | \geq \mbE(N_{GE})/2\} ] \nonumber \\
   &\qquad\mbox{} + \sup_{h: \norm{h}\leq 1} \mbE[ \Delta^2 h(W_d)\,\I\{| N_{GE} - \mbE(N_{GE}) | < \mbE(N_{GE})/2 \}] 
                \nonumber \\
   &\Le 4\mbP[| N_{GE} - n_r \xi | \geq n_r \xi/2 ] \label{Sm1}\\
   &\qquad\mbox{}  +\sup_{h: \norm{h}\leq 1} \mbE\left[ \Delta^2 h(W_d)
            \,\big|\, | N_{GE} - n_r \xi | \leq \frac{n_r \xi}{2} \right].  \label{Sm2}
\end{align}
We establish \eqref{PG2} by separately demonstrating that both \eqref{Sm1} and \eqref{Sm2} are of order~$\bigo(n^{-1})$.

To bound \eqref{Sm1}, observe that, for $i\neq j \in \mathcal{R}$, 
\begin{align*}
   \mbE (\I\{GE_i\} \I\{GE_j\} ) &\Eq \xi \mbE(\I\{GE_j\} |\I\{GE_i\} =1) \\
    &\Eq \xi {n-N^{\mathcal{F}_2}-(d+3) \choose d+3}\left( \frac{|R_i|}{n-U_{I,J}-|R_i|} \right)^{d+3} \\
&\qquad \times \left( \frac{n-U_{I,J}-2|R_i|}{n-U_{I,J}- |R_i|} \right)^{n-N^{\mathcal{F}_2}-2(d+3)} 
         \frac{(d+3)!}{d!} \left( \frac{\pi (r/6)^2}{|R_i|} \right)^{d+3}.
\end{align*}
Under the assumption that $N^{\mathcal{F}_2} \leq \sqrt{n}$, we easily deduce that
\[
    \mbE (\I\{GE_i\} \I\{GE_j\} ) \Eq \xi^2 (1+O(1/n)),
\]
which leads to
\begin{align*}
  \Var(N_{GE}) &\Eq \sum_{i,j} \mbE (\I\{GE_i\} \I\{GE_j\}) - n_r^2 \xi^2 \\ 
      &\Eq n_r(n_r-1) \xi^2(1+O(1/n))+n_r\xi -n_r^2 \xi^2  \\ 
      &\Eq \bigo(n).
\end{align*}
By Chebyshev's inequality, it follows that
\[
   \mbP(|N_{GE} - n_r\xi | > n_r\xi/2 ) \Le \frac{\Var(N_{GE})}{(n_r \xi /2)^2} \Eq \bigo(1/n).
\]

To bound \eqref{Sm2}, we let $X_i = \sum_{j=1}^n \I\{V_j \in R_i \} \I\{M_j= d\}$ be the number of germs in~$R_i$ whose 
grain contains $d$ other germs, and define
\[
    Z_d \Def W_d - \sum_{i \in \mathcal{R}_{GE}} X_i.
\]
For the reasons described above, $\mathcal{L}(X_i \giv i \in \mathcal{R}_{GE}) = \Be(1/2)$ and, given $\mathcal{R}_{GE}$, 
$(X_i)_{i\in \cR_{GE}}$ are conditionally i.i.d.\ and independent of~$Z_d$.

By a standard argument, for $B \sim \Bi(m,1/2)$, $\sup_{h\colon \norm{h}\leq 1}\ex[\Delta^2h(B)] \le Cm^{-1}$,
for a universal constant~$C$.  Since, on $|N_{GE} - n_r\xi| \le n_r\xi/2$, we have $N_{GE} \ge n_r\xi/2 \asymp n$,
it follows that, for all~$h$ with $\norm h \le 1$, we have
\[
     \mbE\left[ \Delta^2 h(W_d)
            \,\big|\, \mathcal{R}_{GE}, | N_{GE} - n_r \xi | \leq \frac{n_r \xi}{2}, Z_d \right] 
         \Eq \bigo(n^{-1}),      
\]
where the constant implied in the order term can be taken to be uniform in~$n$ and in the realizations
of the conditioning random variables.  Taking expectations
establishes~\eqref{PG2}, and thus the lemma in proved in the case $d\geq1$.

The proof in case where $d=0$ follows the same lines, once the definition of the good event in~$R_i$ 
is modified to \emph{(i)} $N_{r/6}(c_{i,1}) =2$, \emph{(ii)} $N_{r/6}(c_{i,2})+N_{r/6}(c_{i,3})=1$, 
\emph{(iii)} $N_{r/6}(c_{i,4}) = 0$, and \emph{(iv)}  $\cV\cap\bclr{R_i \setminus \cup_{j=1}^4 B_{r/6}(c_{i,j})}=\emptyset$.
\end{proof}

In view of Corollary~\ref{GGmoments} and Lemma~\ref{GGL5}, and because~$R=0$ a.s., Theorem~\ref{GGlclt} is proved.

\section{Proof of Theorem~\ref{BRRT}}\label{sec:appenstn}

We begin by giving a high level overview of Stein's method for LCLTs.
Stein's method is used to bound distances that can be expressed in the form
\[
  d_{(*)}\left(\law(W), \law(Z)\right) \Eq \sup_{h \in \mathcal{H}_{(*)}} \left\{ \mbE h(W) - \mbE h(Z) \right\};
\]
here, we consider $\mathcal{H}_{\mathrm{TV}} = \{ \bone_A \colon\, A \subset \mathbb{Z} \}$ and 
$\mathcal{H}_{\mathrm{loc}} = \{ \pm \bone_{\{a\}}\colon\, a \in \mathbb{Z} \}$.
To apply Stein's method, we first find an operator $\mathcal{A}$ such that $\mbE \mathcal{A} f(Z)=0$ for all 
functions~$f$ for which this expectation exists; we then solve the \emph{Stein equation} 
\[
    \mathcal{A} f_h \Eq h - \mbE h(Z),
\]
for all $h \in \mathcal{H}_{(*)}$, yielding the set of solutions~$\mathcal{F}_{(*)}$; finally, we bound 
\begin{equation}\label{SteStep3}
  d_{(*)}(\mathcal{L}(W),\mathcal{L}(Z)) \Eq \sup_{f_h \in \mathcal{F}_{(*)}} \mbE \mathcal{A} f_h(W)
\end{equation}
by deriving properties of the solutions $f_h \in \mathcal{F}_{(*)}$ and by exploiting probabilistic properties of~$W$.
In the case of approximation by translated Poisson distributions, deriving properties of 
$f_h \in \mathcal{F}_{(*)}$ reduces to studying the solutions to the Poisson Stein equation
\begin{equation}\label{PoiSte}
  \lambda f_h(i+1) - i f_h(i) \Eq h(i) - \mbE h(Y), \quad i \geq 0; \quad Y \sim \Po(\lambda).
\end{equation}
Indeed, recalling the definition of the translated Poisson distribution in Section~\ref{SMLCLTs} and
the notation $s$ and~$\gamma$ of~\Ref{TP-notation}, we need only to take 
$\lambda := \sigma^2 + \gamma$ and to replace $f_h(x)$ by $g_h(x)=f_h(x-s)$ for $x \in \mathbb{Z}$.
If $h =\bone_{\{a\}}$ and $f_a$ is the corresponding solution to \eqref{PoiSte}, then the primary property 
of~$f_a$ used to establish the local bound in Theorem~\ref{BRRT} (see \eqref{BRR2}--\eqref{BRR3}) is 
\begin{equation}\label{FDId}
  |\Delta f_a(x) | \Le \frac{1}{\lambda^{3/2} \sqrt{2e}}+ \frac{|\lambda - x|}{\lambda^2} 
       + \frac{1}{\lambda}\bone_{\{a\}}(x),
\end{equation}
where $\Delta$ denotes the first difference operator $\Delta g(k)\Def g(k+1)-g(k)$. We refer the reader 
to \cite[Lemma 3.3]{Bar19} for the derivation of \eqref{FDId}. 

After deriving the necessary properties of $\mathcal{F}_{(*)}$, it still remains to bound \eqref{SteStep3} using 
probabilistic properties of~$W$, for which we use a Stein coupling. 
If $(W,W', G)$ is a Stein coupling and $D:=W'-W$, then
\begin{align}
  | \mbE [\lambda g_h(W+1) &- W g_h(W)] | \label{SteCr1} \\
  &\Eq | \mbE[ \lambda \Delta g_h(W) + G ( g_h(W') - g_h(W))] |  \nonumber \\ 
  &\Le  | \mbE [(\lambda - GD) \Delta g_h(W) ]| + | \mbE [G D \Delta g_h(W) - G( g_h(W') - g_h(W))]|. \label{SteCR}
\end{align}
If we can find a Stein coupling, it allows us to bound \eqref{SteCR}, which is generally easier to bound 
than~\eqref{SteCr1}.
Roughly speaking, if we continue this derivation and apply \eqref{FDId} to each~$g_a$, then we obtain 
the following result, which is relatively straightforward to put together  from \cite{Bar18}.

\begin{theorem}[Corollary 2.3 and Lemma 2.6 of \cite{Bar18}]\label{BRRTdeet}
Let $(W,W',G,R)$ be an approximate Stein coupling with $W$ and~$W'$ integer valued, $\mbE W= \mu$ and $Var(W)= \sigma^2$. 
Set $D := W'-W$, and let $\mathcal{F}_1$ and~$\mathcal{F}_2$ be sigma algebras such that~$W$ is $\mathcal{F}_1$-measurable 
and such that $(G,D)$ is $\mathcal{F}_2$-measurable. Define 
\begin{equation*}
    \Upsilon \Def \mbE\left[| GD(D-1) |\,S_2(\mathcal{L}(W \giv\mathcal{F}_2))\right],
\end{equation*}
and
\[ 
 T:=   \left| \mbE[GD \giv \mathcal{F}_1] - \mbE[GD] \right|. 
\]
Then 
\begin{equation}\label{TVtgb}
   \dtv\left(\mathcal{L}(W), TP(\mu, \sigma^2) \right) \Le
       \frac{1}{\sigma} \left( \s^{-1}\norm{T}_1 + 2 \norm{R}_2 + 2(\Upsilon+1)\right).
\end{equation}
Moreover, for any any positive $t$,
\begin{align}
  \delloc &\Def   \dloc\left(  \mathcal{L} (W), TP(\mu,\sigma^2) \right) \notag \\
    &\Le \frac{1}{\sigma^2} \left(2 \s^{-1}\norm{T}_2
       +  t \left(t^{-1}\mbE\bcls{ T\I[\sigma^{-1} T \geq t]} +
                  \sigma \sup_{a \in \mathbb{Z}} \mbP(W=a) \right) \right) \label{BRR2} \\
&\quad + \frac{\norm{R}_2}{\sigma^2} \left( 3 + \sigma \sup_{a \in \mathbb{Z}} \mbP(W=a) \right) \label{BRR2a} \\
&\quad + \frac{2(\Upsilon+1)}{\sigma^2}. \label{BRR3}
\end{align}
\end{theorem}

\begin{proof}[Proof of Theorem~\ref{BRRT}]
The condition~\eq{eq:upsrti2bd} of Theorem~\ref{BRRT} and the Cauchy--Schwarz inequality imply that~\eq{TVtgb} is 
bounded by $5c_1 \sigma^{-1}$, which is the total variation bound~\eq{eq:ezbd1}. For the local bound, 
it is immediate that~\eq{BRR3} is bounded by $2c_1\sigma^{-2}$. Next, note that
\besn{\label{eq:pwmax}
    \sup_{a\in\IZ} \IP(W=a) 
	&\leq \delloc + \sup_{a\in\IZ} \TP(\mu, \sigma^2)\bclr{\{a\}} \leq \delloc + \s^{-1},
}
and so~\eq{BRR2a} is bounded by $4c_1 \sigma^{-2} + c_1 \s^{-1}\delloc$. For~\eq{BRR2}, 
the first term is easily seen to be bounded by $2c_1\sigma^{-2}$.
To bound the second term, note that Markov's inequality implies that, with $q := \ceil{\log\s}$,
\ben{\label{eq:tilogbal}
 t^{-1}\mbE\bcls{T\I[\sigma^{-1} T \geq t]} \Le \mbE\bcls{ (\sigma^{-1}T)^{q}}\,\frac{\sigma}{t^{q}}
    \Le \sigma (c_2/t)^{q}.
}
Choosing $t=e c_2$ implies that~\eq{eq:tilogbal} is bounded by~$1$.
Hence, invoking~\eq{eq:pwmax}, we have
\be{
 t\left( t^{-1}\mbE\bcls{ T\I[\sigma^{-1} T \geq t]} +\sigma \sup_{a \in \mathbb{Z}} \mbP(W=a) \right) 
 \Le (2 + \s\delloc)ec_2,
}
and hence the second term in~\eq{BRR2} is bounded by $2ec_2\s^{-2} + ec_2\s^{-1}\delloc$.
Since, by Assumption~\Ref{c-condition}, $(c_1 + ec_2)\s^{-1} \le 1/2$, it follows from the bounds
on \Ref{BRR2}--\Ref{BRR3} that
\[
     (1/2)\delloc \Le    \{8c_1 + 2ec_2\}\sigma^{-2},
\]
completing the proof of the local bound~\eq{eq:ezbd}.
\end{proof}

\section{Random sum and binomial moment bounds}\label{sec:branchbinom}

The following two results are \cite[Lemmas~5.1 and~5.2]{Bar18}, and are stated without proof.

\begin{lemma}\label{Branch1}
Let $\mathcal{I}$ be a finite index set, $\mathcal{E}$ be a random (possibly empty) subset of $\mathcal{I}$, and define $E := | \mathcal{E} |$. Let $\{ Y_i \}_{i \in \mathcal{I}}$ be a collection of random variables independent of $\mathcal{E}$ and, for $\ell \in \mbN$, let $\max_{i \in \mathcal{I}} \norm{ Y_i }_\ell \leq y$. Then 
\[
\norm{\sum_{i \in \mathcal{E}} Y_j}_\ell \leq y \norm{E}_\ell.
\]
\end{lemma}
\begin{lemma}\label{BinMom}
Let $n \in \mbN$, $0 \leq p \leq 1$, $Y \sim Bi(n,p)$, and $\ell \in \mbN$. Then $\norm{Y}_\ell \leq A(np, \ell)$, where
\[
A(x,\ell) := \pi e^{e-2} \times 
\begin{cases}
\ell / \log \left( (e-1) \right), \quad & \ell >x, \\
x, & \ell \leq x.
\end{cases}
\]
In particular, $A(x,l) \leq C_A(x \vee l ) \leq C_A (x + l)$, where $C_A := \pi e^{e-2} / \log(e-1)$.
\end{lemma}

\section{Acknowledgments}

The authors were partially supported by the Australian Research Council (ARC) Grant DP150101459 and 
the ARC Centre of Excellence for Mathematical and Statistical Frontiers, CE140100049. PB was partially supported by 
ARC Grant FL130100039. We thank the referees for their comments.

%

\end{document}